\documentclass[12pt]{article}
\usepackage{mathrsfs}
\usepackage{amssymb}
\usepackage[hang,flushmargin]{footmisc}
\usepackage{multirow}
\usepackage{amsfonts}
\usepackage{graphicx}
\usepackage{subfigure}
\usepackage{amsmath,amsthm,amssymb,latexsym,color}
\usepackage[colorlinks, linkcolor=blue, anchorcolor=gree, citecolor=red]{hyperref}
\usepackage[numbers,sort&compress]{natbib}
\usepackage{xcolor}
\usepackage{enumerate}
\usepackage{bm}
 \makeatletter
    
    \newcommand{\Rmnum}[1]
    {\expandafter\@slowromancap\romannumeral #1@}
    \makeatother
\def\wz{\tilde}

\newtheorem{thm}{Theorem}[section]

\newtheorem{prop}[thm]{Proposition}
\newtheorem{lemma}[thm]{Lemma}

\newcounter{fooo}[section]
\newtheorem{step}[fooo]{Step}

\newtheorem{example}[thm]{Example}

\newtheorem{defin}[thm]{Definition}

\theoremstyle{definition}
\newtheorem{remark}[thm]{Remark}

\usepackage{CJK}
\begin{document}
\begin{CJK*}{GBK}{song}

\renewcommand{\baselinestretch}{1.3}
\title{Weakly distance-regular digraphs whose underlying graphs are distance-regular, I}

\author{Yuefeng Yang\textsuperscript{a} ~ Qing Zeng\textsuperscript{b}\footnote{\scriptsize Corresponding author.}  ~~  Kaishun Wang\textsuperscript{b}
\\
{\footnotesize  \textsuperscript{a} \em School of Science, China University of Geosciences, Beijing, 100083, China} \\
{\footnotesize  \textsuperscript{b} \em  Laboratory of Mathematics and Complex Systems (MOE),}  \\
{\footnotesize  \em School of Mathematical Sciences, Beijing Normal University, Beijing, 100875, China } }

\date{}
\maketitle
\footnote{\scriptsize\qquad{\em E-mail address:} ~yangyf@cugb.edu.cn~~(Yuefeng Yang),~~qingz@mail.bnu.edu.cn ~~(Qing Zeng),~~ wangks@bnu.edu.cn ~~(Kaishun Wang).}

\begin{abstract}

Weakly distance-regular digraphs are a natural directed version of distance-regular graphs. In \cite{KSW03}, the third  author and Suzuki proposed a question when an orientation of a distance-regular graph defines a weakly distance-regular digraph. In this paper, we initiate this project, and classify all commutative weakly distance-regular digraphs whose underlying graphs are Hamming graphs, folded $n$-cubes and Doob graphs, respectively.
\medskip


\end{abstract}

\section{Introduction}

A \emph{digraph} $\Gamma$ is a pair $(V(\Gamma),A(\Gamma))$ where $V(\Gamma)$ is a finite nonempty set of vertices and $A(\Gamma)$ is a set of ordered pairs ({\em arcs}) $(x,y)$ with distinct vertices $x$ and $y$. For any arc $(x,y)\in A(\Gamma)$, if $A(\Gamma)$ also contains an arc $(y,x)$, then $\{(x,y),(y,x)\}$ can be viewed as an {\em edge}. We say that $\Gamma$ is an \emph{undirected graph} or a {\em graph} if $A(\Gamma)$ is a symmetric relation.  A vertex $x$ is {\em adjacent} to $y$ if $(x,y)\in A(\Gamma)$. In this case, we also call $y$ an \emph{out-neighbour} of $x$, and $x$ an \emph{in-neighbour} of $y$. For every vertex $x$ of $\Gamma$, we denote the number of in-neighbours (resp. out-neighbours) of $x$ by $d^{-}_{\Gamma}(x)$ (resp. $d^{+}_{\Gamma}(x)$). A digraph is said to be \emph{regular of valency $k$} if the number of in-neighbours and out-neighbours of all vertices are all equal to $k$.

A \emph{path} of length $r$ from $x$ to $y$ in $\Gamma$ is a finite sequence of vertices $(x=w_{0},w_{1},\ldots,w_{r}=y)$ such that $(w_{t-1}, w_{t})\in A(\Gamma)$ for $1\leq t\leq r$. A digraph (resp. graph) is said to be \emph{strongly connected} (resp. \emph{connected}) if, for any two vertices $x$ and $y$, there is a path from $x$ to $y$. A path $(w_{0},w_{1},\ldots,w_{r-1})$ is called a \emph{circuit} of length $r$ when $(w_{r-1},w_0)\in A\Gamma$. The \emph{girth} of $\Gamma$ is the length of a shortest circuit in $\Gamma$.

The length of a shortest path from $x$ to $y$ is called the \emph{distance} from $x$ to $y$ in $\Gamma$, denoted by $\partial_\Gamma(x,y)$. The maximum value of the distance function in $\Gamma$ is called the \emph{diameter} of $\Gamma$. We define $\Gamma_{i}$ ($0\leq i\leq d$) to be the set of ordered pairs $(x,y)$ with $\partial_{\Gamma}(x,y)=i$, where $d$ is the diameter of $\Gamma$. Let $\wz{\partial}_{\Gamma}(x,y):=(\partial_{\Gamma}(x,y),\partial_{\Gamma}(y,x))$ be the \emph{two-way distance} from $x$ to $y$, and $\wz{\partial}(\Gamma):=\{\wz{\partial}_{\Gamma}(x,y)\mid x,y\in V(\Gamma)\}$ the \emph{two-way distance set} of $\Gamma$. For any $\wz{i}\in\wz{\partial}(\Gamma)$, we define $\Gamma_{\wz{i}}$ to be the set of ordered pairs $(x,y)$ with $\wz{\partial}_{\Gamma}(x,y)=\wz{i}$. An arc $(x,y)$ of $\Gamma$ is of type $(1,r)$ if $\partial_{\Gamma}(y,x)=r$.


In \cite{KSW03}, the third  author and Suzuki proposed a natural directed version of a distance-regular graph without bounded diameter, i.e., a weakly distance-regular digraph. A strongly connected digraph $\Gamma$ is said to be \emph{weakly distance-regular} if, for any $\tilde{h},\tilde{i},\tilde{j}\in\tilde{\partial}(\Gamma)$, the number of $z\in V(\Gamma)$ such that $\tilde{\partial}_{\Gamma}(x,z)=\tilde{i}$ and $\tilde{\partial}_{\Gamma}(z,y)=\tilde{j}$ is constant whenever $\tilde{\partial}_{\Gamma}(x,y)=\tilde{h}$. This constant is denoted by $p^{\tilde{h}}_{\tilde{i},\tilde{j}}(\Gamma)$. The integers $p^{\tilde{h}}_{\tilde{i},\tilde{j}}(\Gamma)$ are called the \emph{intersection numbers} of $\Gamma$. If no confusion occurs, we write $p^{\tilde{h}}_{\tilde{i},\tilde{j}}$ instead of $p^{\tilde{h}}_{\tilde{i},\tilde{j}}(\Gamma)$. We say that $\Gamma$ is \emph{commutative} if $p^{\tilde{h}}_{\tilde{i},\tilde{j}}=p^{\tilde{h}}_{\tilde{j},\tilde{i}}$ for all $\tilde{h},\tilde{i},\tilde{j}\in\tilde{\partial}(\Gamma)$.  An
important property of weakly distance-regular digraphs is that $\mathfrak{X}(\Gamma)=(V(\Gamma),\{\Gamma_{\wz{i}}\}_{\wz{i}\in\wz{\partial}(\Gamma)})$ is a non-symmetric association scheme. We call $\mathfrak{X}(\Gamma)$ the \emph{attached scheme} of $\Gamma$. We will introduce the definition of association schemes and distance-regular graphs in Section 2.  For more information about weakly distance-regular digraphs, see \cite{YF22,HS04,KSW04,YYF16,YYF18,YYF20,YYF22,ZYW22}.


For a digraph $\Gamma$, we form the {\em underlying graph} of $\Gamma$ with the same vertex set, and there is an edge between vertices $x$ and $y$ whenever $(x,y)\in A(\Gamma)$ or $(y,x)\in A(\Gamma)$. A digraph is {\em semicomplete} if its underlying graph is a complete graph. A semicomplete weakly distance-regular digraph has diameter $2$ and girth $g\leq3$ (See Proposition \ref{complete}).
In fact, 
if $\Gamma$ is a semicomplete weakly distance-regular digraph of girth $2$ (resp. $3$), then $\mathfrak{X}(\Gamma)$ is a non-symmetric association scheme with $3$ (resp. $2$) classes, and every non-symmetric association scheme with $3$ (resp. $2$) classes arises from a semicomplete weakly distance-regular digraph of girth $2$ (resp. $3$) in this way.

For the digraphs $\Gamma$ and $\Sigma$, the {\em Cartesian product} $\Gamma\,\square\,\Sigma$ is the digraph with the vertex set $V(\Gamma)\times V(\Sigma)$ such that $((u,v),(u',v'))$ is an arc if and only if $u=u'$ and $(v,v')\in A(\Sigma)$, or $(u,u')\in A(\Gamma)$ and $v=v'$.

In \cite{KSW03}, the third  author and Suzuki proposed a question when an orientation of a distance-regular graph defines a weakly distance-regular digraph.   In this paper, we initiate this project, classifying all commutative weakly distance-regular digraphs
whose underlying graphs are Hamming graphs, folded n-cubes, and
Doob graphs, respectively. In the forthcoming paper \cite{Z}, we shall determine commutative weakly distance-regular digraphs whose underlying graphs are Johnson graphs, folded Johnson graphs, Odd graphs,
and doubled Odd graphs.


Note that Hamming graphs can be viewed as Cartesian product of complete graphs.
Our first main theorem characterizes  the case that the underlying graphs are Hamming graphs.



\begin{thm}\label{main}
Let $\Gamma$ be a commutative weakly distance-regular digraph. Then $\Gamma$ has a Hamming graph as its underlying graph if and only if $\Gamma$ is isomorphic to one of the following digraphs:
\begin{itemize}
\item[{\rm (i)}] ${\rm Cay}(\mathbb{Z}_4,\{1\})$;

\item[{\rm (ii)}] ${\rm Cay}(\mathbb{Z}_4\times\mathbb{Z}_2,\{(1,0),(0,1)\})$;

\item[{\rm (iii)}] $\Delta^1$ or $\Delta^{1}\,\square\,\Delta^{2}$;

\item[{\rm (iv)}] $\Gamma^{1}\,\square\,\Gamma^{2}\,\square\cdots\square\,\Gamma^{d}$.
\end{itemize}
Here, $d\geq1$, and $(\Delta^{i})_{i\in\{1,2\}}$~(resp. $(\Gamma^i)_{i\in\{1,2,\ldots,d\}}$) are semicomplete weakly distance-regular digraphs of diameter $2$ and girth $2$ (resp. $3$) with $p^{\tilde{h}}_{\tilde{i},\tilde{j}}(\Delta^1)=p^{\tilde{h}}_{\tilde{i},\tilde{j}}(\Delta^2)$ (resp. $p^{\tilde{h}}_{\tilde{i},\tilde{j}}(\Gamma^1)=p^{\tilde{h}}_{\tilde{i},\tilde{j}}(\Gamma^l)$) for each $l$ and $\tilde{h},\tilde{i},\tilde{j}$.
\end{thm}

\begin{remark}
The underlying graphs of $\Gamma$ in Theorem \ref{main} are  $H(2,2)$ for (i), $H(3,2)$ for (ii), $H(1,q)$ or $H(2,q)$ with $q>2$ for (iii), and $H(d,q)$ with $q\equiv3~\mbox{mod}~4$ for (iv).
\end{remark}

Note that  folded $n$-cubes and Doob graphs are related to Hamming graphs. The following two main theorems determine  the cases that the underlying graphs are folded $n$-cubes and Doob graphs, respectively.



\begin{thm}\label{main3}
Let $\Gamma$ be a commutative weakly distance-regular digraph. Then $\Gamma$ dose not have a folded $n$-cube as its underlying graph if $n\geq3$.
\end{thm}


\begin{thm}\label{main2}
Let $\Gamma$ be a commutative weakly distance-regular digraph. Then $\Gamma$ has a Doob graph as its underlying graph if and only if $\Gamma$ is isomorphic to {\rm Cay}$(\mathbb{Z}_4\times\mathbb{Z}_4,\{(1,0),(0,1),(-1,-1)\})$.
\end{thm}

%
%
%



The remainder of this paper is organized as follows. In Section 2, we provide the required notation, concepts and preliminary results for association schemes and distance-regular graphs. In Section 3, we give some results  concerning subdigraphs of a class of special weakly distance-regular digraphs. In Sections 4--6, we prove Theorems \ref{main}, \ref{main3} and \ref{main2} based on the results in Section 3, respectively.

\section{Preliminaries}

In this section, we present some notation, concepts, and basic results for association schemes and distance-regular graphs that we shall use in the remainder of the paper.

\subsection{Association schemes}

A \emph{$d$-class association scheme} $\mathfrak{X}$ is a pair $(X,\{R_{i}\}_{i=0}^{d})$, where $X$ is a finite set, and each $R_{i}$ is a
nonempty subset of $X\times X$ satisfying the following axioms (see \cite{EB84,PHZ96,PHZ05} for a background of the theory of association schemes):
\begin{itemize}
\item [{\rm(i)}] $R_{0}=\{(x,x)\mid x\in X\}$ is the diagonal relation;

\item [{\rm(ii)}] $X\times X=R_{0}\cup R_{1}\cup\cdots\cup R_{d}$, $R_{i}\cap R_{j}=\emptyset~(i\neq j)$;

\item [{\rm(iii)}] for each $i$, $R_{i}^{\rm T}=R_{i^{*}}$ for some $0\leq i^{*}\leq d$, where $R_{i}^{\rm T}=\{(y,x)\mid(x,y)\in R_{i}\}$;

\item [{\rm(iv)}] for all $i,j,l$, the cardinality of the set $$P_{i,j}(x,y):=R_{i}(x)\cap R_{j^{*}}(y)$$ is constant whenever $(x,y)\in R_{l}$, where $R(x)=\{y\mid (x,y)\in R\}$ for $R\subseteq X\times X$ and $x\in X$. This constant is denoted by $p_{i,j}^{l}$.
\end{itemize}
A $d$-class association scheme is also called an association scheme with $d$ classes (or even simply a scheme). The integers $p_{i,j}^{l}$ are called the \emph{intersection numbers} of $\mathfrak{X}$. We say that $\mathfrak{X}$ is \emph{commutative} if $p_{i,j}^{l}=p_{j,i}^{l}$ for all $0\leq i,j,l\leq d$. The subsets $R_{i}$ are called the \emph{relations} of $\mathfrak{X}$. For each $i$, the integer $k_{i}=p_{i,i^{*}}^{0}$ is called the \emph{valency} of $R_{i}$.   A relation $R_{i}$ is called \emph{symmetric} if $i=i^{*}$, and \emph{non-symmetric} otherwise. An association scheme is called \emph{symmetric} if all relations are symmetric, and \emph{non-symmetric} otherwise.


We close this subsection with two basic properties of intersection
numbers which are used frequently in this paper.

\begin{lemma}\label{intersection numners}
{\rm (\cite[Chapter \Rmnum{2}, Proposition 2.2]{EB84})} Let $(X,\{R_{i}\}_{i=0}^{d})$ be an association scheme. The following hold{\rm:}
\begin{itemize}
\item [{\rm(i)}] $k_{i}k_{j}=\sum_{l=0}^{d}p_{i,j}^{l}k_{l}$;

\item [{\rm(ii)}] $p_{i,j}^{l}k_{l}=p_{l,j^{*}}^{i}k_{i}=p_{i^{*},l}^{j}k_{j}$.
%
%
\end{itemize}
\end{lemma}

\subsection{Distance-regular graphs}

A connected graph $\Gamma$ is said to be \emph{distance-transitive} if, for any vertices $x$, $y,$ $x'$ and $y'$ of $\Gamma$ satisfying $\partial_{\Gamma}(x,y)=\partial_{\Gamma}(x',y')$, there exists an automorphism $\sigma$ of $\Gamma$ such that $x'=\sigma(x)$ and $y'=\sigma(y)$.
A connected graph $\Gamma$ of diameter $d$ is said to be \emph{distance-regular} if there are integers $b_{i-1}$, $c_i$ for $1\leq i\leq d$ such that there are exactly $c_i$ neighbours of $y$ in $\Gamma_{i-1}(x)$ and $b_i$ neighbours of $y$ in $\Gamma_{i+1}(x)$ for any $(x,y)\in\Gamma_{i}$. This definition implies the existence of numbers $a_i$ such that there are exactly $a_i$ neighbours of $y$ in $\Gamma_{i}(x)$ for $0\leq i\leq d$. They do not depend on $x$ and $y$ because $b_0$ is the valency of $\Gamma$ and $a_i$
is determined in terms of $b_i$ and $c_i$ by the equation $b_{0}=c_{i}+a_{i}+b_{i}$, where $c_0=b_d=0$. We observe $a_0=0$ and $c_1=1$. The array $\iota(\Gamma)=\{b_0,b_1,\ldots,b_{d-1}; c_1,c_2,\ldots,c_d\}$ is called the \emph{intersection array} of $\Gamma$. An important property of distance-regular graphs is that $(V(\Gamma),\{\Gamma_{i}\}_{i=0}^{d})$ is a symmetric association scheme. One can easily observe that $a_{i}=p_{1,i}^{i}$, $b_{i}=p_{1,i+1}^{i}$ and $c_{i}=p_{1,i-1}^{i}$ for $0\leq i\leq d$. For more detailed results on such special graphs, we refer the reader to \cite{NB74,AEB98,DKT16}.

The Hamming graph is a classical example of distance-regular graphs, and they naturally model Hamming distance. They feature in different areas of computer science and mathematics, including coding theory, algebraic graph theory and model theory. The definition of a  Hamming graph is as follows. For positive integers $d,q$ with $q>1$ and a set $S$ with $q$ elements, the {\em Hamming graph} $H(d,q)$ is the graph whose vertices correspond to the elements of $S^d$, where there is an edge between two vertices whenever they agree in all but one coordinate.
Hamming graphs also can be viewed as Cartesian products of complete graphs.
The following theorem determines the intersection array of a Hamming graph.

\begin{thm}\label{hamming}
{\rm (\cite[Theorem 9.2.1]{AEB98})} Let $d$ and $q$ be two positive integers with $q>1$. Then the Hamming graph $H(d,q)$ is  a distance-regular graph of diameter $d$, and has intersection array given by $b_i=(d-i)(q-1)$, $c_i=i$ and $a_i=i(q-2)$ for $0\leq i\leq d$.
\end{thm}

The Hamming graph $H(n,2)$ is also called a \emph{(hyper)cube} or the $n$-\emph{cube}. Its folded graph is called a {\em folded $n$-cube} $\square_n$. In other word, for integer $n\geq2$, the folded $n$-cube is the graph $H(n-1,2)$ with a perfect matching introduced between antipodal vertices, where two vertices in $\square_n$ are called the {\em antipodal vertices} if they differ in all coordinates.  
The following theorem determines the intersection array of a folded $n$-cube for $n\geq3$.

\begin{thm}\label{cube}
{\rm (\cite[Section 9.2]{AEB98})} Let $n\geq3$. Then the folded $n$-cube $\square_n$ is  a distance-regular graph of diameter $\lfloor n/2\rfloor$, and has intersection array given by $b_i=n-i$, $c_i=i$ and $a_i=0$ for $2i<n$, and if $n$ is even, then $c_d=n$.
\end{thm}

The Hamming graph $H(d,q)$ is characterized by its intersection array unless $q=4$ and $d>1$, in which case there are also the so-called Doob graphs. For $d_1>0$ and $d_2\geq0$, the {\em Doob graph} $G(d_1,d_2)$ is the Cartesian product of $H(d_2,4)$ with $d_1$ copies of the  Shrikhande graph, where the  \emph{Shrikhande graph} is isomorphic to ${\rm Cay}(\mathbb{Z}_4\times\mathbb{Z}_4,\{(\pm1,0),(0,\pm1),\pm(1,1)\})$. Note that $G(d_1,d_2)$ has the same intersection numbers as $H(2d_1+d_2,4)$.

%

\begin{thm}\label{Doob}
{\rm (\cite[Section 9.2]{AEB98})} Let $d_1>0$ and $d_2\geq0$. Then the Doob graph $G(d_1,d_2)$ is  a distance-regular graph of diameter $2d_1+d_2$, and has intersection array given by $b_i=3(2d_1+d_2-i)$, $c_i=i$ and $a_i=2i$ for $0\leq i\leq 2d_1+d_2$.
\end{thm}

\section{General results}

In the remainder of this paper, we always assume that $\Gamma$ is a commutative weakly distance-regular digraph with  vertex set $S^{d}$, where $S$ is a set of $q$ elements, and that the underlying graph of $\Gamma$ is a distance-regular graph $\Sigma$.

To distinguish the notation between $\Gamma$ and $\Sigma$, let $p_{\wz{i},\wz{j}}^{\wz{h}}, P_{\wz{i},\wz{j}}(x,y)$ denote the symbols belonging to $\Gamma$ with $\wz{i},\wz{j},\wz{h}\in\wz{\partial}(\Gamma)$ and $x,y\in V(\Gamma)$, and $a_i,c_i$ the symbols belonging to $\Sigma$ with $0\leq i\leq d'$, where $d'$ is the diameter of $\Sigma$. For $(a,b)\in\wz{\partial}(\Gamma)$, we write $\Gamma_{a,b}$ instead of $\Gamma_{(a,b)}$.

First, we give some general results about weakly distance-regular digraphs whose underlying graphs are distance-regular graphs.

\begin{lemma}\label{s=t}
Suppose that $c_2=2$ and $p_{(1,p-1),(1,s-1)}^{(2,2)}\neq0$.  Then $p=s$. Moreover, exactly one of the following holds:
\begin{itemize}
\item[{\rm(i)}] $p=2$ and $p_{(1,1),(1,1)}^{(2,2)}=2$;

\item[{\rm(ii)}] $p\in\{3,4\}$ and $p_{(1,p-1),(1,p-1)}^{(2,2)}=1$.
\end{itemize}
\end{lemma}
\begin{proof}
Let $(x,z)\in\Gamma_{2,2}$.
Then $(x,z)\in\Sigma_2$. The fact $c_2=2$ implies that there exist vertices $y$ and $y'$ such that $\Sigma_1(x)\cap\Sigma_1(z)=\{y,y'\}$. Since $p_{(1,p-1),(1,s-1)}^{(2,2)}\neq0$, we may assume $y\in P_{(1,p-1),(1,s-1)}(x,z)$.

Suppose $p\neq s$. It follows that $y'\in P_{(1,p-1),(1,s-1)}(z,x)$. By the commutativity of $\Gamma$, there exists a vertex $y''\in P_{(1,s-1),(1,p-1)}(x,z)$ with $y''\notin\{y,y'\}$, which implies $y''\in\Sigma_1(x)\cap\Sigma_1(z)$, contrary to the fact that $c_2=2$. The first statement is valid.

Since $p_{(1,p-1),(1,p-1)}^{(2,2)}\neq0$, one has $p\in\{2,3,4\}$.  If $p\in\{3,4\}$, then $y'\in P_{(1,p-1),(1,p-1)}(z,x)$, and so $P_{(1,p-1),(1,p-1)}(x,z)=\{y\}$. Thus, (ii) is valid.

Now we consider the case $p=2$. If $y'\in P_{(1,l-1),(h-1,1)}(x,z)$ for some $l>2$ or $h>2$, from the commutativity of $\Gamma$, then there exists   $y''\in P_{(h-1,1),(1,l-1)}(x,z)$ with $y''\notin\{y,y'\}$, which implies $y,y',y''\in\Sigma_1(x)\cap\Sigma_1(z)$,  contrary to the fact that $c_2=2$. Hence, $y'\notin P_{(1,l-1),(h-1,1)}(x,z)\cup P_{(h-1,1),(1,l-1)}(x,z)$ for  $l>2$ or $h>2$. By the first statement, one has $y'\in P_{(1,l-1),(1,l-1)}(x,z)\cup P_{(1,l-1),(1,l-1)}(z,x)$. If $l>2$, then there exists $y''\in P_{(1,l-1),(1,l-1)}(z,x)\cup P_{(1,l-1),(1,l-1)}(x,z)$ with $y''\notin\{y,y'\}$, which implies $y,y',y''\in\Sigma_1(x)\cap\Sigma_1(z)$, a contradiction. Hence, $l=2$. Since $\{y,y'\}=P_{(1,1),(1,1)}(x,z),$ (i) is valid.
%
\end{proof}

\begin{lemma}\label{pneqt}
Suppose that  $c_2=2$ and $p_{(2,2),(p-1,1)}^{(1,p-1)}\neq0$. Then $p_{(2,2),(r-1,1)}^{(1,t-1)}=0$ for all $t\neq p$.
\end{lemma}
\begin{proof}
According to Lemma \ref{intersection numners} (ii), one has $p_{(1,p-1),(1,p-1)}^{(2,2)}\neq0$. Note that $c_2=2$. If $p=2$, by Lemma \ref{s=t} (i), then $p_{(1,1),(1,1)}^{(2,2)}=2$, which implies $p_{(1,t-1),(1,r-1)}^{(2,2)}=0$ for all $t\neq2$; if $p>2$,
by Lemma \ref{intersection numners} (ii), then $p_{(p-1,1),(p-1,1)}^{(2,2)}\neq0$, which implies $p_{(1,t-1),(1,r-1)}^{(2,2)}=0$ for all $t\neq p$. From Lemma \ref{intersection numners} (ii), one gets $p_{(2,2),(r-1,1)}^{(1,t-1)}=0$ for all $t\neq p$.
\end{proof}

\begin{lemma}\label{0}
Suppose that $a_1=2$.  If $s>2$, then $p^{(1,s-1)}_{(1,p-1),(p-1,1)}=0$.
\end{lemma}
\begin{proof}
Suppose for the contrary that $p^{(1,s-1)}_{(1,p-1),(p-1,1)}\neq0$ for some $(1,p-1)\in\tilde{\partial}(\Gamma)$. By Lemma \ref{intersection numners} (ii) and the commutativity of $\Gamma$, we have $p^{(1,p-1)}_{(s-1,1),(1,p-1)}=p^{(1,p-1)}_{(1,p-1),(s-1,1)}=p^{(1,p-1)}_{(1,s-1),(1,p-1)}\neq0$, contrary to the fact that $a_1=2$.
\end{proof}

For each $i\in\{1,2,\ldots,d\}$ and $a_j\in S$ with $1\leq j\leq d-1$, denote $\Gamma_{i}(a_{1},a_{2},\ldots,a_{d-1})$ be the induced subdigraph of $\Gamma$ on the set
\begin{align}
\{(a_1,a_2,\ldots,a_{i-1},b,a_{i},a_{i+1},\ldots,a_{d-1})\mid b\in S\}.\nonumber
\end{align}

In the sequel, we denote $a_{[i]}:=(\underbrace{a,\ldots,a}_{i})$ for $i\geq1$ and $a\in S$. Fix $o\in S$. For $1\leq i\leq d$,  define the digraph $\Gamma^{[i]}$ with vertex set $S^{i}$ and $(\alpha,\beta)$ is an arc of $\Gamma^{[i]}$ whenever $((\alpha,o_{[d-i]}),(\beta,o_{[d-i]}))$ is an arc of $\Gamma$.

In the remainder of this section, we always assume $a_1=q-2$, $c_2=2$, and the underlying graph of $\Gamma^{[m]}$ is a Hamming graph  for some positive integer $m$. Set $T=S^{m-1}\times\{o_{[d-m]}\}$. Note that $\Gamma_i(\alpha)$ is a semicomplete digraph for $1\leq i\leq m$ and $\alpha\in T$. Moreover, $\Gamma_i(\alpha)$ may be a complete graph.

We prove the following results under the above assumptions.

\begin{lemma}\label{proof1}
Let $1\leq i\leq m$ and $\alpha\in T$. Suppose $d_{\Gamma_i(\alpha)}^+(x)=d_{\Gamma_i(\alpha)}^+(y)$ or $d_{\Gamma_i(\alpha)}^-(x)=d_{\Gamma_i(\alpha)}^-(y)$ for  $x,y\in V(\Gamma_i(\alpha))$. If $\Gamma_i(\alpha)$ is not complete, then  $\Gamma_i(\alpha)$ is  a weakly distance-regular digraph of diameter $2$, girth $g$ and $\wz{\partial}(\Gamma_i(\alpha))=\{(0,0),(1,2),(2,1),(1,g-1)\}$  with $g\leq3$.
\end{lemma}
\begin{proof}
Let $(u,v)$ be an arc of $\Gamma_i(\alpha)$. Set $\partial_{\Gamma}(v,u)=p-1$ with $p>1$. Since the proofs are similar, without loss of generality, we may assume $d_{\Gamma_i(\alpha)}^+(u)=d_{\Gamma_i(\alpha)}^+(v)$. Suppose that $p>3$.
Since  $\Gamma_i(\alpha)$ is  semicomplete,  each out-neighbour of $v$ in $V(\Gamma_i(\alpha))$  is also an out-neighbour of $u$, which implies $d_{\Gamma_i(\alpha)}^+(u)\geq d_{\Gamma_i(\alpha)}^+(v)+1$, a contradiction.
Hence, $p=2$ or $3$.

Since $\Gamma_i(\alpha)$ is semicomplete, one gets $V(\Gamma_i(\alpha))\setminus\{u,v\}\subseteq\Sigma_1(u)\cap\Sigma_1(v)$. By $a_1=q-2$, we obtain   $V(\Gamma_i(\alpha))\setminus\{u,v\}=\Sigma_1(u)\cap\Sigma_1(v)$.  If $p=2$, then $(u,v)$ is also an edge in $\Gamma_i(\alpha)$; if $p=3$, then there exists $w\in V(\Gamma)$ such that $(v,w,u)$ is a path of $\Gamma$, which implies that $(v,w,u)$ is also a path of $\Gamma_i(\alpha)$ since $\Sigma_1(u)\cap\Sigma_1(v)\subseteq V(\Gamma_i(\alpha))$. Since $\Gamma_i(\alpha)$ is  semicomplete and $(u,v)\in A(\Gamma_i(\alpha))$ was arbitrary, we have $\wz{\partial}_{\Gamma}(x,y)=\wz{\partial}_{\Gamma_i(\alpha)}(x,y)$ for all $x,y\in V(\Gamma_i(\alpha))$. Note that $\Gamma_i(\alpha)$ is not complete. Then $$\wz{\partial}(\Gamma_i(\alpha))=\{(0,0),(1,2),(2,1)\}~\mbox{or}~ \{(0,0),(1,1),(1,2),(2,1)\}.$$ The fact that $\Sigma_1(x)\cap\Sigma_1(y)\subseteq V(\Gamma_i(\alpha))$ for $x,y\in V(\Gamma_i(\alpha))$ implies that $[\Gamma_i(\alpha)]_{\wz{i}}(x)\cap[\Gamma_i(\alpha)]_{\wz{j}^*}(y)=P_{\wz{i},\wz{j}}(x,y)$ for $\wz{i},\wz{j}\in\wz{\partial}(\Gamma_i(\alpha))$. By the weakly distance-regularity of $\Gamma$, $\Gamma_i(\alpha)$ is a weakly distance-regular digraph of diameter $2$ and girth $g$ with $g\leq3$.
\end{proof}

\begin{prop}\label{complete}
Let $\Gamma$ be a semicomplete weakly distance-regular digraph of diameter $d$ and girth $g$.  Then $d=2$, $g\leq3$ and $\wz{\partial}(\Gamma)=\{(0,0),(1,2),(2,1)\}$ or $\{(0,0),(1,1),(1,2),(2,1)\}.$
\end{prop}
\begin{proof}
Since $\Gamma$ is regular, from Lemma \ref{proof1}, the desired result follows.
\end{proof}

\begin{lemma}\label{jb4}
Let $x\in V(\Gamma)\setminus V(\Gamma_i(\alpha))$ for $1\leq i\leq m$ and $\alpha\in T$. Then $x$ has at most one neighbour  in $V(\Gamma_i(\alpha))$.
\end{lemma}
\begin{proof}
Suppose   $u,v\in\Sigma_1(x)\cap V(\Gamma_i(\alpha))$. Note that $\Gamma_i(\alpha)$ is semicomplete. Then $\Sigma_1(u)\cap\Sigma_1(v)\supseteq (V(\Gamma_{i}(\alpha))\setminus\{u,v\})\cup\{x\}$. Since $a_1=q-2$, we have $u=v$.
\end{proof}

\begin{lemma}\label{degree-1}
Let $x,y\in V(\Gamma_i(\alpha))$  and $z\in V(\Gamma)\setminus V(\Gamma_i(\alpha))$ for $1\leq i\leq m$ and $\alpha\in T$. Suppose $(x,y)\in\Gamma_{1,p-1}$ for $p\geq2$. If $(x,z)\in\Gamma_{1,s-1}$ for $s\geq2$, then there exists a unique vertex $y'\in P_{(1,s-1),(p-1,1)}(y,z)$ with $y'\notin V(\Gamma_i(\alpha))$.
\end{lemma}
\begin{proof}
In view of  Lemma \ref{jb4}, we have $(y,z)\in\Sigma_2$.  Since $c_2=2$, there exists a unique vertex $y'$ such that $\Sigma_1(y)\cap\Sigma_1(z)=\{x,y'\}$. By Lemma \ref{jb4}, we obtain  $y'\notin V(\Gamma_i(\alpha))$.
Observe that $x\in P_{(p-1,1),(1,s-1)}(y,z)$. If $p>2$ or $s>2$, by the commutativity of $\Gamma$, then $y'\in P_{(1,s-1),(p-1,1)}(y,z)$; if $p=s=2$, by $(y,z)\in\Sigma_2$, then $(y,z)\in\Gamma_{2,2}$, which implies $y'\in P_{(1,1),(1,1)}(y,z)$ from Lemma \ref{s=t} (i). 
Thus, the desired result is valid.
\end{proof}

\begin{lemma}\label{degree-2}
Let $x,y\in V(\Gamma_i(\alpha))$ and $z\in V(\Gamma)\setminus V(\Gamma_i(\alpha))$ for $1\leq i\leq m$ and $\alpha\in T$ . Suppose $p=2$ or $(x,z)\notin\Gamma_{2,2}$. Then the following hold:
\begin{itemize}
\item[{\rm(i)}] If  $(x,y)\in\Gamma_{1,p-1}$ and $(y,z)\in A(\Gamma)$, then there exists a unique vertex $y'\in V(\Gamma)\setminus V(\Gamma_i(\alpha))$ such that $(x,y',z)$ is a path in $\Gamma$;

\item[{\rm(ii)}] If  $(y,x)\in\Gamma_{1,p-1}$ and $(z,y)\in A(\Gamma)$, then there exists a unique vertex $y'\in V(\Gamma)\setminus V(\Gamma_i(\alpha))$ such that $(z,y',y)$ is a path in $\Gamma$.
\end{itemize}
\end{lemma}
\begin{proof}
(i)~In view of  Lemma \ref{jb4}, we have $(x,z)\in\Sigma_2$.  Since $c_2=2$, there exists a unique vertex $y'$ such that $\Sigma_1(x)\cap\Sigma_1(z)=\{y,y'\}$. By Lemma \ref{jb4}, we obtain  $y'\notin V(\Gamma_i(\alpha))$.  Suppose $(x,z)\notin\Gamma_{2,2}$. Since $(x,z)\in\Sigma_2$, we have $(x,z)\in\Gamma_{2,b}$ with $b>2$, which implies $p_{(r-1,1),(s-1,1)}^{(2,b)}=0$ for all $r,s>1$. Note that $p_{(1,p-1),(1,p'-1)}^{(2,b)}\neq0$ for $\partial_{\Gamma}(z,y)=p'-1$. If $p_{(1,r-1),(s-1,1)}^{(2,b)}\neq0$ for some $r,s>2$,  by  the commutativity of $\Gamma$, then $p_{(s-1,1),(1,r-1)}^{(2,b)}=p_{(1,r-1),(s-1,1)}^{(2,b)}\neq0$, contrary to the fact that $c_2=2$. It follows that $p_{(1,r-1),(s-1,1)}^{(2,b)}=p_{(s-1,1),(1,r-1)}^{(2,b)}=0$ with $r,s>2$. Then $\bigcup_{r,s>1}P_{(1,r-1),(1,s-1)}(x,z)=\{y,y'\}$. Suppose $(x,z)\in\Gamma_{2,2}$. Then $p=2$. By Lemma \ref{s=t}, we obtain  $P_{(1,1),(1,1)}(x,z)=\{y,y'\}$. Thus, the desired result is valid.

(ii)~The proof is similar to  (i), hence omitted.
\end{proof}


\begin{lemma}\label{degree-3}
Let $x,y\in V(\Gamma_i(\alpha))$ and $z\in V(\Gamma)\setminus V(\Gamma_i(\alpha))$ for $1\leq i\leq m$ and $\alpha\in T$ . Suppose  $(x,y)\in\Gamma_{1,p-1}$ for $p>2$. 
If $(y,z)\in\Gamma_{1,s-1}$ with $(x,z)\in\Gamma_{2,2}$ for $s\geq2$, then $s=p$ and there exists a unique vertex $y'\in P_{(1,p-1),(1,p-1)}(z,x)$ with $y'\notin V(\Gamma_i(\alpha))$.
\end{lemma}
\begin{proof}
In view of  Lemma \ref{jb4}, we have $(x,z)\in\Sigma_2$.  Since $c_2=2$, there exists a unique vertex $y'$ such that $\Sigma_1(x)\cap\Sigma_1(z)=\{y,y'\}$. By Lemma \ref{jb4}, we obtain  $y'\notin V(\Gamma_i(\alpha))$. Note that $y\in P_{(1,p-1),(1,s-1)}(x,z)$.
According to Lemma \ref{s=t}, we have $s=p$ and $p_{(1,p-1),(1,p-1)}^{(2,2)}=1$. Then $y'\in P_{(1,p-1),(1,p-1)}(z,x)$. Thus, this lemma is valid.
%
\end{proof}

In Lemmas \ref{jb4}--\ref{degree-3}, we summarize the relations between the vertices in $V(\Gamma_i(\alpha))$ and the vertices in $V(\Gamma)\setminus V(\Gamma_i(\alpha))$. Proposition  \ref{degree-main} will be used frequently in the proof of our main results, and its proof relies on Lemmas \ref{jb4}--\ref{degree-3}. 


\begin{prop}\label{degree-main}
Let $x,y\in V(\Gamma_i(\alpha))$ for $1\leq i\leq m$ and $\alpha\in T$. Suppose $(x,y)\in\Gamma_{1,p-1}$ for $p\geq2$. Then, exactly one of the following holds:
\begin{itemize}
\item[{\rm(i)}] $d^{+}_{\Gamma_i(\alpha)}(x)=d^{+}_{\Gamma_i(\alpha)}(y)$, $d^{-}_{\Gamma_i(\alpha)}(x)=d^{-}_{\Gamma_i(\alpha)}(y)$, and $p=2$ or $p_{(2,2),(s-1,1)}^{(1,p-1)}=0$ for all $(1,s-1)\in\wz{\partial}(\Gamma)$;

\item[{\rm(ii)}] $d^{+}_{\Gamma_i(\alpha)}(x)>d^{+}_{\Gamma_i(\alpha)}(y)$, $d^{-}_{\Gamma_i(\alpha)}(x)<d^{-}_{\Gamma_i(\alpha)}(y)$, and $p\neq2$ and $p_{(2,2),(s-1,1)}^{(1,p-1)}\neq0$ for some $(1,s-1)\in\wz{\partial}(\Gamma)$.
\end{itemize}
\end{prop}
\begin{proof}


Since $c_2=2$, from Lemmas \ref{jb4} and \ref{degree-1},  there exists a mapping $\varphi$ from $\Gamma_1(x)\setminus V(\Gamma_i(\alpha))$ to $\Gamma_1(y)\setminus V(\Gamma_i(\alpha))$ such that $(x,\varphi(w))\in\Sigma_2$ and $\Sigma_1(x)\cap\Sigma_1(\varphi(w))=\{y,w\}$. Let $w_1,w_2\in\Gamma_1(x)\setminus V(\Gamma_i(\alpha))$.  If $\varphi(w_1)=\varphi(w_2)$,  then $y,w_1,w_2\in\Sigma_1(x)\cap\Sigma_1(\varphi(w_1))$, and so $w_1=w_2$ since $c_2=2$. Thus, $\varphi$ is an injection.


\textbf{Case 1.} $p=2$ or $p_{(2,2),(s-1,1)}^{(1,p-1)}=0$ for all $(1,s-1)\in\wz{\partial}(\Gamma)$.

Let $w'\in\Gamma_1(y)\setminus V(\Gamma_i(\alpha))$. Lemma \ref{jb4} implies $(x,w')\in\Sigma_2$. By
Lemma \ref{degree-2} (i), there exists $w\in\Gamma_1(x)\setminus V(\Gamma_i(\alpha))$ such that $\Sigma_1(x)\cap\Sigma_1(w')=\{w,y\}$. Thus, $w'\in{\rm Im}(\varphi)$, and so $\varphi$ is a surjection. It follows that $\varphi$ is a bijection. Then $|\Gamma_1(x)\setminus V(\Gamma_i(\alpha))|=|\Gamma_1(y)\setminus V(\Gamma_i(\alpha))|$, and so $d^+_{\Gamma}(x)-d^+_{\Gamma_i(\alpha)}(x)=d^+_{\Gamma}(y)-d^+_{\Gamma_i(\alpha)}(y)$. Hence, $d^+_{\Gamma_i(\alpha)}(x)=d^+_{\Gamma_i(\alpha)}(y)$.  By the similar argument, one gets $d^-_{\Gamma_i(\alpha)}(x)=d^-_{\Gamma_i(\alpha)}(y)$. Therefore, (i) holds.


\textbf{Case 2.} $p\neq2$ and $p_{(2,2),(s-1,1)}^{(1,p-1)}\neq0$ for some $(1,s-1)\in\wz{\partial}(\Gamma)$.

Pick a vertex $w'\in P_{(2,2),(s-1,1)}(x,y)$. Then $(x,w')\in\Sigma_2$. Since $\Gamma_{i}(\alpha)$ is semicomplete, one has $w'\in\Gamma_1(y)\setminus V(\Gamma_i(\alpha))$. Since $c_2=2$, there exists a unique vertex $w$ such that $\Sigma_1(x)\cap\Sigma_1(w')=\{y,w\}$. In view of Lemma \ref{degree-3}, $w$ is an in-neighbour of $x$. Thus, $w'\notin{\rm Im}(\varphi)$, and so $\varphi$ is not a surjection. It follows that $|\Gamma_1(x)\setminus V(\Gamma_i(\alpha))|<|\Gamma_1(y)\setminus V(\Gamma_i(\alpha))|$, and so $d^+_{\Gamma}(x)-d^+_{\Gamma_i(\alpha)}(x)<d^+_{\Gamma}(y)-d^+_{\Gamma_i(\alpha)}(y)$. Hence, $d^+_{\Gamma_i(\alpha)}(x)>d^+_{\Gamma_i(\alpha)}(y)$.  By the similar argument, one gets $d^-_{\Gamma_i(\alpha)}(x)<d^-_{\Gamma_i(\alpha)}(y)$. Therefore, (ii) holds.
%
\end{proof}

\begin{lemma}\label{p=4}
Let $x,y\in V(\Gamma_i(\alpha))$ for $1\leq i\leq m$ and $\alpha\in T$.    If $d_{\Gamma_i(\alpha)}^+(x)>d_{\Gamma_i(\alpha)}^+(y)$, then $(x,y)\in\Gamma_{1,3}$ and $p_{(2,2),(3,1)}^{(1,3)}\neq0$.
\end{lemma}
\begin{proof}
Note that $\Gamma_i(\alpha)$ is semicomplete. Since  $d_{\Gamma_i(\alpha)}^+(x)>d_{\Gamma_i(\alpha)}^+(y)$, by Proposition \ref{degree-main}, one obtains $(x,y)\in\Gamma_{1,p-1}$ for $p>2$ and $p_{(2,2),(s-1,1)}^{(1,p-1)}\neq0$ for some $(1,s-1)\in\wz{\partial}(\Gamma)$. Lemma \ref{intersection numners} (ii) implies that $p_{(1,p-1),(1,s-1)}^{(2,2)}\neq0$. In view of Lemma \ref{s=t}, we obtain $s=p$ and $p\in\{3,4\}$.

Suppose $p=3$. By Lemma \ref{jb4}, there exists a path $(y,z,x)$ in $\Gamma_i(\alpha)$ with $(y,z),(z,x)\in\Gamma_{1,1}\cup\Gamma_{1,2}$. In view of Proposition \ref{degree-main}, one gets $$d^{+}_{\Gamma_i(\alpha)}(x)>d^{+}_{\Gamma_i(\alpha)}(y)\geq d^{+}_{\Gamma_i(\alpha)}(z)\geq d^{+}_{\Gamma_i(\alpha)}(x),$$ a contradiction. Hence, $p=4$, and so $p_{(2,2),(3,1)}^{(1,3)}\neq0$.
\end{proof}

\begin{prop}\label{degree*}
Let $q=2$. If $\Sigma$ is distance-transitive, then $p_{(2,2),(3,1)}^{(1,3)}\neq0$ and each arc of $\Gamma$ is of type $(1,1)$ or $(1,3)$. In particular, $k_{1,1}=0$ or $1$.
\end{prop}
\begin{proof}
Let $(x,y)\in\Gamma_{1,p-1}$ with $p>2$. Then $(x,y)\in\Sigma_1$.  Since $\Sigma$ is distance-transitive, we may assume $x,y\in V(\Gamma_i(\alpha))$ for some $i\in\{1,2,\ldots,m\}$ and $\alpha\in T$.
The fact that $q=2$ implies $a_1=0$ and $d^+_{\Gamma_i(\alpha)}(x)>d^+_{\Gamma_i(\alpha)}(y)$. By Lemma \ref{p=4}, one gets $p=4$ and $p_{(2,2),(3,1)}^{(1,3)}\neq0$. Thus, each arc of $\Gamma$ is of type $(1,1)$ or $(1,3)$. 
%
The first statement is valid.

Now suppose $(1,1)\in\wz{\partial}(\Gamma)$. Since $p^{(1,3)}_{(2,2),(3,1)}\neq0$, by Lemma \ref{pneqt}, we have $p_{(1,1),(2,2)}^{(1,1)}=0$, and so $p_{(1,1),(1,1)}^{(2,2)}=0$ from Lemma \ref{intersection numners} (ii). Note that $a_1=0$. In view of  Lemma \ref{intersection numners} (i), we obtain $k_{1,1}^2=p^{(0,0)}_{(1,1),(1,1)}k_{1,1}$. Thus, $k_{1,1}=1$. The second statement is also valid.
\end{proof}

%

\section{Proof of Theorem \ref{main}}


To prove Theorem \ref{main}, we need some auxiliary results.


\begin{prop}\label{sufficient}
Let $\Delta$ and $\Delta'$ be two semicomplete weakly distance-regular digraphs of girth $2$ with the same intersection numbers. Then $\Delta\square\Delta'$ is weakly distance-regular.
\end{prop}
\begin{proof}
Let $\Lambda=\Delta\square\Delta'$. Then
\begin{align}
\wz{\partial}_{\Lambda}((x,x'),(y,y'))=\wz{\partial}_{\Delta}(x,y)+\wz{\partial}_{\Delta'}(x',y').\label{distance}
\end{align}
Note that $\Delta$ and $\Delta'$ are two semicomplete weakly distance-regular digraphs of girth $2$ and diameter $2$. Then $\wz{\partial}(\Delta)=\wz{\partial}(\Delta')=\{(0,0),(1,2),(2,1),(1,1)\}$. According to \eqref{distance}, one gets $\wz{\partial}(\Lambda)=\{\wz{i}+\wz{j}\mid \wz{i},\wz{j}\in\wz{\partial}(\Delta)\}$. 

Set $D_{\wz{a}+\wz{b}}=\{(\wz{a},\wz{b}),(\wz{b},\wz{a})\}$ for all $\wz{a},\wz{b}\in\{(0,0),(1,1),(1,2),(2,1)\}$. For each $\wz{h}\in\wz{\partial}(\Lambda)$, there exists exactly one set $D_{\wz{a}+\wz{b}}$  such that $\wz{a}+\wz{b}=\wz{h}$. So we write $D_{\wz{h}}$ instead of $D_{\wz{a}+\wz{b}}$. Fix $\wz{h}\in\wz{\partial}(\Lambda)$. Let $((u,u'),(v,v'))\in\Lambda_{\wz{h}}$.
Note that $(\wz{\partial}_{\Delta}(u,v),\wz{\partial}_{\Delta'}(u',v'))\in D_{\wz{h}}$.  By \eqref{distance}, we have
\begin{align}
\Lambda_{\wz{i}}(u,u')\cap\Lambda_{\wz{j}^*}(v,v')=\bigcup_{(\wz{a},\wz{b})\in D_{\wz{i}},(\wz{c},\wz{d})\in D_{\wz{j}}}(\Delta_{\wz{a}}(u)\cap\Delta_{\wz{c}^*}(v))\times(\Delta_{\wz{b}}'(u')\cap\Delta_{\wz{d}^*}'(v'))\nonumber
\end{align}
for all $\wz{i},\wz{j}\in\wz{\partial}(\Lambda)$. Since $\Delta$ and $\Delta'$ are weakly distance-regular digraphs with the same intersection numbers, the cardinality of the set $\Lambda_{\wz{i}}(u,u')\cap\Lambda_{\wz{j}^*}(v,v')$ does not depend on the choice of $(u,u'),(v,v')$ but only on $\wz{i},\wz{j},\wz{h}\in\wz{\partial}(\Lambda)$. Since $\wz{h}$ was arbitrary, we  complete the proof.
\end{proof}

In the remainder of this section, let $\Sigma$ be the Hamming graph $H(d,q)$ on the set $S$ of $q$ elements with $d>0$ and $q>1$.

\begin{lemma}\label{match}
Let $1\leq i\leq d$ and $\alpha,\beta\in S^{d-1}$ with $\alpha\neq\beta$. Suppose that $x,y\in V(\Gamma_i(\alpha))$ and $z\in V(\Gamma_i(\beta))$ such that $x\neq y$ and $(y,z)\in\Sigma_1$. If $y'\in\Sigma_1(x)\cap\Sigma_1(z)$ with $y'\neq y$, then $y'\in V(\Gamma_i(\beta))$.
\end{lemma}
\begin{proof}
Without loss of generality, we may assume $i=d$ and $x=(\alpha,a)$ for some $a\in S$. The fact that $y\in V(\Gamma_d(\alpha))$ with $y\neq x$ implies that there exists $b\in S\setminus\{a\}$ such that $y=(\alpha,b)$. Since $z\in\Gamma_1(y)\cap V(\Gamma_d(\beta))$ with $\alpha\neq\beta$, one gets $z=(\beta,b)$, and $\alpha,$ $\beta$ agree in all but one coordinate. Since $\Sigma$ is a Hamming graph, we obtain $(x,z)\in\Sigma_2$ and $\Sigma_1(x)\cap\Sigma_1(z)=\{y,(\beta,a)\}$ from Theorem \ref{hamming}. It follows that $y'=(\beta,a)$, and so $y'\in V(\Gamma_d(\beta))$.
\end{proof}

\begin{lemma}\label{(2,2)}
Let $1\leq i\leq d$ and $\alpha,\beta\in S^{d-1}$ with $\alpha\neq\beta$. Suppose that $x,y\in V(\Gamma_i(\alpha))$ and $x',y'\in V(\Gamma_i(\beta))$ with $x\neq y$ and $x'\neq y'$. If $(x,x'),(y',y)\in\Gamma_{1,p-1}$ for $p>2$, then $(x',y'),(y,x)\in\Gamma_{1,p-1}$.
\end{lemma}
%
\begin{proof}
By Theorem \ref{hamming}, we get $c_2=2$ and $a_1=q-2$. Note that $\Gamma_i(\alpha)$ and $\Gamma_i(\beta)$ are semicomplete. Since $x\neq y$ and $x'\neq y'$, one gets $(x,y)$, $(x',y')\in\Sigma_1$.
In view of Lemma \ref{jb4}, we have $(x',y)\in\Sigma_2$, and so $\Sigma_1(x')\cap\Sigma_1(y)=\{x,y'\}$ since $c_2=2$.
If $(x,y)\in\Gamma_{1,s-1}$ for some $s\geq2$,
by Lemma \ref{degree-1}, then $y'\in P_{(1,p-1),(s-1,1)}(y,x')$ since $(x,x')\in\Gamma_{1,p-1}$, contrary to the fact that $p>2$. Thus,   $(y,x)\in\Gamma_{1,s-1}$ for some $s>2$. Since $x\in P_{(1,s-1),(1,p-1)}(y,x')$ and $(y',y)\in\Gamma_{1,p-1}$, from Lemma \ref{degree-2} (i), we obtain $(y,x')\in\Gamma_{2,2}$. By Lemma \ref{degree-3}, one has  $s=p$ and $(x',y')\in\Gamma_{1,p-1}$.
\end{proof}

The above two lemmas summarize the relations between the vertices in $V(\Gamma_i(\alpha))$ and the vertices in $V(\Gamma_i(\beta))$ with $\alpha\neq\beta$.

\begin{lemma}\label{degree}
Let  $1\leq i\leq d$ and $\alpha\in S^{d-1}$. If $q>2$,  then $d^{+}_{\Gamma_i(\alpha)}(x)=d^+_{\Gamma_i(\alpha)}(y)$ and $d^{-}_{\Gamma_i(\alpha)}(x)=d^-_{\Gamma_i(\alpha)}(y)$ for all $x,y\in V(\Gamma_i(\alpha))$.
\end{lemma}
\begin{proof}
Suppose for the contrary that $d_{\Gamma_i(\alpha)}^+(x)\neq d_{\Gamma_i(\alpha)}^+(y)$ or $d_{\Gamma_i(\alpha)}^-(x)\neq d_{\Gamma_i(\alpha)}^-(y)$ for some $x,y\in V(\Gamma_i(\alpha))$. From Theorem \ref{hamming}, we have $a_1=q-2$ and $c_2=2$. Since $\Gamma_i(\alpha)$ is semicomplete, we may assume $(x,y)\in\Gamma_{1,p-1}$ with $p>1$. In view of Proposition \ref{degree-main}, one has $d^{+}_{\Gamma_i(\alpha)}(x)>d^{+}_{\Gamma_i(\alpha)}(y)$ and $d^{-}_{\Gamma_i(\alpha)}(x)<d^{-}_{\Gamma_i(\alpha)}(y)$. By Lemma \ref{p=4}, one gets $p=4$ and $p_{(2,2),(3,1)}^{(1,3)}\neq0$. It follows from Lemma \ref{intersection numners} (ii) and Lemma \ref{s=t} that $p_{(1,3),(1,3)}^{(2,2)}=1$.

Since $\Gamma_i(\alpha)$ is semicomplete and $p_{(2,2),(3,1)}^{(1,3)}\neq0$, there exists a vertex $y'\in P_{(2,2),(3,1)}(x,y)$ with $y'\notin V(\Gamma_i(\alpha))$. Let $y'\in V(\Gamma_i(\beta))$ for some $\beta\in S^{d-1}\setminus\{\alpha\}$. By Lemma \ref{degree-3}, there exists a unique vertex $x'\in P_{(1,3),(1,3)}(y',x)$. Lemma \ref{match} implies that $x'\in V(\Gamma_i(\beta))$.


Suppose $(x,z),(y,z)\in\Gamma_{1,3}\cup\Gamma_{3,1}$ for all $z\in V(\Gamma_i(\alpha))$. If $(y,z)\in\Gamma_{1,3}$, then $(x,z)\in\Gamma_{1,3}$, which implies $x\in P_{(3,1),(1,3)}(y,z)$; if $(z,y)\in\Gamma_{1,3}$, then $x\in P_{(3,1),(1,3)}(z,y)$ or $z\in P_{(3,1),(1,3)}(x,y)$. Thus, $p^{(1,3)}_{(3,1),(1,3)}\neq0$. Since $\Gamma_i(\alpha)$ is semicomplete and $q>2$, there exists $z\in P_{(3,1),(1,3)}(x,y)$. Lemma \ref{jb4} implies $$|P_{(3,1),(1,3)}(x,y)|=d^-_{\Gamma_i(\alpha)}(x)\geq1+|P_{(1,3),(1,3)}(z,x)|+|P_{(3,1),(1,3)}(z,x)|,$$ a contradiction. Therefore, $(x,z)\notin\Gamma_{1,3}\cup\Gamma_{3,1}$ or $(y,z)\notin\Gamma_{1,3}\cup\Gamma_{3,1}$ for some $z\in V(\Gamma_i(\alpha))$.


Since $\Gamma_i(\alpha)$ is semicomplete, by the commutativity of $\Gamma$, we may assume $(y,z)\in\Gamma_{1,t-1}$ or $(z,y)\in\Gamma_{1,t-1}$ for some $z\in V(\Gamma_{i}(\alpha))$ with $t\neq 4$.
By Lemma \ref{pneqt}, we get $p_{(2,2),(r-1,1)}^{(1,t-1)}=0$ for all $(1,r-1)\in\wz{\partial}(\Gamma)$. In view of Proposition \ref{degree-main} (i), one obtains $d_{\Gamma_i(\alpha)}^+(y)=d_{\Gamma_i(\alpha)}^+(z)$ and $d_{\Gamma_i(\alpha)}^-(y)=d_{\Gamma_i(\alpha)}^-(z)$. Since $d^{+}_{\Gamma_i(\alpha)}(x)>d^{+}_{\Gamma_i(\alpha)}(y)$ and $d^{-}_{\Gamma_i(\alpha)}(x)<d^{-}_{\Gamma_i(\alpha)}(y)$, from Lemma \ref{p=4}, we have $(x,z)\in\Gamma_{1,3}$.
%
Since $x\in P_{(3,1),(1,3)}(y,z)$, from Lemma \ref{intersection numners} (ii), one gets $p_{(1,3),(t-1,1)}^{(1,3)}\neq0$. Without loss of generality, we assume $z\in P_{(1,3),(t-1,1)}(x,y)$.

Note that $y'\in V(\Gamma_i(\beta))$. Since $y\in P_{(t-1,1),(1,3)}(z,y')$,
by the commutativity of $\Gamma$, there exists a vertex $z'\in P_{(1,3),(t-1,1)}(z,y')$. Lemma \ref{match} implies that $z'\in V(\Gamma_i(\beta))$. Since $(x',x)\in\Gamma_{1,3}$, from Lemma \ref{(2,2)}, we obtain $(z',x')\in\Gamma_{1,3}$. By Lemma \ref{jb4}, we have $(y,z')\in\Sigma_2$. Since $(z',x',x,y)$ is a path and $p_{(2,2),(r-1,1)}^{(1,t-1)}=0$ for all $(1,r-1)\in\wz{\partial}(\Gamma)$, one gets $(y,z')\in\Gamma_{2,3}$.

Since $x\in P_{(3,1),(1,3)}(y,z)$, from the commutativity of $\Gamma$, there exists $w\in P_{(1,3),(3,1)}(y,z)$. In view of Lemma \ref{jb4}, we have $w\in V(\Gamma_i(\alpha))$. Since  $z\in P_{(3,1),(1,3)}(z',w)$,  there exists  $w'\in P_{(1,3),(3,1)}(z',w)$. It follows from Lemma \ref{match} that $w'\in V(\Gamma_i(\beta))$. 
Since $(x',x)\in\Gamma_{1,3}$, from Lemma \ref{(2,2)}, we get $(w',x')\in\Gamma_{1,3}$.

From Lemma \ref{jb4}, we have $(z,w')\in\Sigma_2$. Note that $w,z'\in P_{(1,3),(1,3)}(z,w')$ and $p_{(1,3),(1,3)}^{(2,2)}=1$. The fact that $(w',x',x,z)$ is a path implies  $(z,w')\in\Gamma_{2,3}$. Since $z\in P_{(1,t-1),(1,3)}(y,z')$,  there exists  $w''\in P_{(1,t-1),(1,3)}(z,w')$ such that $\Sigma_1(z)\cap\Sigma_1(w')\supseteq\{z',w,w''\}$, contrary to the fact that $c_2=2$.
\end{proof}

\begin{lemma}\label{q=2-1}
If $q=2$, then $k_{1,3}=1$.
\end{lemma}
\begin{proof}
Let $S=\{o,a\}$.   According to Theorem \ref{hamming}, we have $a_1=0$ and $c_2=2$. By \cite[Theorem 7.5.2]{AEB98}, $\Sigma$ is distance-transitive.
From   Proposition \ref{degree*}, without loss of generality, we may assume that $(o_{[d]},(a,o_{[d-1]}),(a_{[2]},o_{[d-2]}),(o,a,o_{[d-2]}))$ is a circuit consisting of arcs of type $(1,3)$.

Assume for contrary, namely, $k_{1,3}>1$.
Then $d\geq3$.   Without loss of generality, we assume  $(o_{[d]},(o_{[2]},a,o_{[d-3]}))\in\Gamma_{1,3}$. Since $\Sigma$ is a Hamming graph, we get $((a,o_{[d-1]}),(o_{[2]},a,o_{[d-3]}))\in\Sigma_2$ and $\Sigma_1(a,o_{[d-1]})\cap\Sigma_1(o_{[2]},a,o_{[d-3]})=\{o_{[d]},(a,o,a,o_{[d-3]})\}$. Note that $o_{[d]}\in P_{(3,1),(1,3)}((a,o_{[d-1]}),(o_{[2]},a,o_{[d-3]}))$. The commutativity of $\Gamma$ implies $(a,o,a,o_{[d-3]})\in P_{(1,3),(3,1)}((a,o_{[d-1]}),(o_{[2]},a,o_{[d-3]}))$. By the similar argument, one has $$(u,a,o_{[d-3]})\in P_{(1,3),(3,1)}((u,o_{[d-2]}),(v,a,o_{[d-3]}))$$ for each $(u,v)\in\{(a_{[2]},(a,o)),((o,a),a_{[2]}),(o_{[2]},(o,a))\}.$  Since $$\Sigma_1(o_{[d]})\cap\Sigma_1(o,a_{[2]},o_{[d-3]})=\{(o,a,o_{[d-2]}),(o_{[2]},a,o_{[d-3]})\},$$ one gets $(o_{[d]},(o,a_{[2]},o_{[d-3]}))\in\Gamma_{4,4}$.

Since $(o,a,o_{[d-2]}),(a_{[3]},o_{[d-3]})\in P_{(1,3),(1,3)}((a_{[2]},o_{[d-2]}),(o,a_{[2]},o_{[d-3]})),$ from Lemma \ref{s=t} (ii), one gets $((a_{[2]},o_{[d-2]}),(o,a_{[2]},o_{[d-3]}))\notin\Gamma_{2,2}.$
By Proposition \ref{degree*}, one obtains $p_{(2,2),(3,1)}^{(1,3)}\neq0$.  Then  $d\geq4$. Without loss of generality, we may assume $(a_{[4]},o_{[d-4]})\in P_{(2,2),(3,1)}((a_{[2]},o_{[d-2]}),(a_{[3]},o_{[d-3]})).$ Since $\Sigma$ is the underlying graph of $\Gamma$, we get $(o_{[d]},(a_{[2]},o_{[d-2]}))\in\Gamma_{2,2}$, which implies $(o_{[d]},(a_{[4]},o_{[d-4]}))\in\Gamma_{4,4}$.
One can verify that  $P_{(1,3),(3,1)}(o_{[d]},(a_{[4]},o_{[d-4]}))=\emptyset$ and $(o_{[2]},a,o_{[d-3]})\in P_{(1,3),(3,1)}(o_{[d]},(o,a_{[2]},o_{[d-3]}))$, a contradiction.
\end{proof}

In the notation in Section 3,   recall that the digraph $\Gamma^{[i]}$ with vertex set $S^{i}$ and $(\alpha,\beta)$ is an arc of $\Gamma^{[i]}$ whenever $((\alpha,o_{[d-i]}),(\beta,o_{[d-i]}))$ is an arc of $\Gamma$ for $1\leq i\leq d$. Define the digraph $\Gamma^i$ with vertex set $S$, and $(a,b)$ is an arc of $\Gamma^{i}$ if and only if $((o_{[i-1]},a,o_{[d-i]}),(o_{[i-1]},b,o_{[d-i]}))$ is an arc of $\Gamma$, where $1\leq i\leq d$.

\begin{lemma}\label{product}
If $q>2$, then $\Gamma^{[i]}=\Gamma^{[i-1]}\square\Gamma^{i}$ for $2\leq i\leq d$.
\end{lemma}
\begin{proof}
Let $\alpha,\beta\in S^{i-1}$. In view of Theorem \ref{hamming}, we have $a_1=q-2$ and $c_2=2$. Note that $((\alpha,e),(\beta,f))\in A(\Gamma^{[i]})$ if and only if $((\alpha,e,o_{[d-i]}),(\beta,f,o_{[d-i]}))\in A(\Gamma)$. Since $\Sigma$ is a Hamming graph, $((\alpha,e),(\beta,f))\in A(\Gamma^{[i]})$ if and only if $e=f$ and $((\alpha,e),(\beta,e))\in A(\Gamma^{[i]})$, or $\alpha=\beta$  and $((\alpha,e),(\alpha,f))\in A(\Gamma^{[i]})$ for all $\alpha,\beta\in S^{i-1}$ and $e,f\in S$. We prove this lemma step by step.


\begin{step}\label{e=f}
{\rm For each $e\in S$,  $(\alpha,\beta)\in A(\Gamma^{[i-1]})$ whenever $((\alpha,e),(\beta,e))\in A(\Gamma^{[i]})$.}
\end{step}

Note that $(\alpha,\beta)\in A(\Gamma^{[i-1]})$ if and only if $((\alpha,o_{[d-i+1]}),(\beta,o_{[d-i+1]}))\in A(\Gamma)$. Let $a\in S$ with $o\neq a$.
Observe that $\Gamma_i(\alpha,o_{[d-i]})$ and $\Gamma_i(\beta,o_{[d-i]})$ are semicomplete. Then $((\alpha,o_{[d-i+1]}),(\alpha,a,o_{[d-i]}))$, $((\beta,o_{[d-i+1]}),(\beta,a,o_{[d-i]}))\in\Sigma_1$. 
 Without loss of generality, we may assume $((\alpha,o_{[d-i+1]}),(\alpha,a,o_{[d-i]}))\in A(\Gamma).$
Since $\Sigma$ is a Hamming graph, one gets $((\alpha,o_{[d-i+1]}),(\beta,o_{[d-i+1]}))\in\Sigma_1$ if and only if $((\alpha,a,o_{[d-i]}),(\beta,a,o_{[d-i]}))\in\Sigma_1$.

Suppose that $((\alpha,o_{[d-i+1]}),(\beta,o_{[d-i+1]}))$ is an arc in $\Gamma$. It follows that $((\alpha,a,o_{[d-i]}),(\beta,a,o_{[d-i]}))\in\Sigma_1$.   In view of Lemma \ref{jb4}, one has $(\beta,o_{[d-i+1]})\in\Sigma_2(\alpha,a,o_{[d-i]})$. The fact $c_2=2$ implies $\Sigma_1(\beta,o_{[d-i+1]})\cap\Sigma_1(\alpha,a,o_{[d-i]})=\{(\alpha,o_{[d-i+1]}),(\beta,a,o_{[d-i]})\}$. By Lemma \ref{degree-1}, we get $((\alpha,a,o_{[d-i]}),(\beta,a,o_{[d-i]}))\in A(\Gamma)$.

Suppose that $((\alpha,a,o_{[d-i]}),(\beta,a,o_{[d-i]}))$ is an arc in $\Gamma$. It follows that $((\alpha,o_{[d-i+1]}),(\beta,o_{[d-i+1]}))\in\Sigma_1$. In view of Lemma \ref{jb4}, we have $(\beta,a,o_{[d-i]})\in\Sigma_2(\alpha,o_{[d-i+1]})$. The fact $c_2=2$ implies $\Sigma_1(\alpha,o_{[d-i+1]})\cap\Sigma_1(\beta,a,o_{[d-i]})=\{(\alpha,a,o_{[d-i]}),(\beta,o_{[d-i+1]})\}$. According to Proposition \ref{degree-main} and Lemma  \ref{degree}, one obtains $((\alpha,o_{[d-i+1]}),(\alpha,a,o_{[d-i]}))\in\Gamma_{1,1}$ or $((\alpha,o_{[d-i+1]}),(\beta,a,o_{[d-i]}))\notin\Gamma_{2,2}$. In view of Lemma \ref{degree-2} (i), we get $((\alpha,o_{[d-i+1]}),(\beta,o_{[d-i+1]}))\in A(\Gamma)$.

By above argument, $((\alpha,o_{[d-i+1]}),(\beta,o_{[d-i+1]}))$ is an arc in $\Gamma$ if and only if $((\alpha,a,o_{[d-i]}),(\beta,a,o_{[d-i]}))$ is an arc in $\Gamma$. Since $a\in S\setminus\{o\}$ was arbitrary, we complete the proof of this step.

\begin{step}\label{alpha}
{\rm For each $\alpha\in S^{i-1}$, $(e,f)\in A(\Gamma^{i})$ whenever $((\alpha,e),(\alpha,f))\in A(\Gamma^{[i]})$.}
\end{step}

Let $\gamma\in S^{i-1}$. Note that there exists a sequence $(\alpha_0=\alpha,\alpha_1,\ldots,\alpha_l=\gamma)$ of $V(\Gamma^{[i-1]})$ such that $(\alpha_j,\alpha_{j+1})$ or $(\alpha_{j+1},\alpha_{j})\in A(\Gamma^{[i-1]})$ with $0\leq j\leq l-1$. By Step \ref{e=f}, $((\alpha_j,a),(\alpha_{j+1},a))$ or $((\alpha_{j+1},a),(\alpha_{j},a))\in A(\Gamma^{[i]})$ with $0\leq j\leq l-1$ for all $a\in S$. Therefore, $((\alpha_j,a,o_{[d-i]}),(\alpha_{j+1},a,o_{[d-i]}))$ or $((\alpha_{j+1},a,o_{[d-i]}),(\alpha_{j},a,o_{[d-i]}))\in A(\Gamma)$ with $0\leq j\leq l-1$ for all $a\in S$.

Observe that $\Gamma_i(\alpha_0,o_{[d-i]})$ and $\Gamma_i(\alpha_1,o_{[d-i]})$ are semicomplete digraphs. Then $((\alpha_k,e,o_{[d-i]}),(\alpha_{k},f,o_{[d-i]}))\in\Sigma_1$ for $e\neq f$ and $k\in\{0,1\}$. Without loss of generality, we may  assume
$((\alpha_0,e,o_{[d-i]}),(\alpha_0,f,o_{[d-i]}))\in A(\Gamma)$.  Note that $((\alpha_{0},a,o_{[d-i]}),(\alpha_{1},a,o_{[d-i]}))\in\Sigma_1$ for all $a\in S$. 
In view of Lemma \ref{jb4}, we get  $((\alpha_1,e,o_{[d-i]}),(\alpha_{0},f,o_{[d-i]}))\in\Sigma_2$, and so $\Sigma_1(\alpha_1,e,o_{[d-i]})\cap\Sigma_1(\alpha_{0},f,o_{[d-i]})=\{(\alpha_{0},e,o_{[d-i]}),(\alpha_1,f,o_{[d-i]})\}$ since $c_2=2$.
If $((\alpha_0,e,o_{[d-i]}),(\alpha_1,e,o_{[d-i]}))\in A(\Gamma)$,
by Lemma \ref{degree-1}, then $((\alpha_1,e,o_{[d-i]}),(\alpha_1,f,o_{[d-i]}))\in A(\Gamma)$.  Now we suppose that $((\alpha_1,e,o_{[d-i]}),(\alpha_{0},e,o_{[d-i]}))\in A(\Gamma)$. By Proposition \ref{degree-main} and Lemma \ref{degree}, we get $((\alpha_0,e,o_{[d-i]}),(\alpha_0,f,o_{[d-i]}))\in\Gamma_{1,1}$ or $((\alpha_1,e,o_{[d-i]}),(\alpha_{0},f,o_{[d-i]}))\notin\Gamma_{2,2}$, which implies $((\alpha_1,e,o_{[d-i]}),(\alpha_1,f,o_{[d-i]}))\in A(\Gamma)$ from Lemma \ref{degree-2} (ii).  By repeating this process, we get $((\alpha_j,e,o_{[d-i]}),(\alpha_j,f,o_{[d-i]}))\in A(\Gamma)$ for all $1\leq j\leq l$, and so $((\gamma,e),(\gamma,f))\in A(\Gamma^{[i]})$.

Since $\alpha,\gamma\in S^{i-1}$ were arbitrary, $((o_{[i-1]},e),(o_{[i-1]},f))\in A(\Gamma^{[i]})$ if and only if $((\alpha,e),(\alpha,f))\in A(\Gamma^{[i]})$ for all $\alpha\in S^{i-1}$. Observe that $(e,f)\in A(\Gamma^{i})$ whenever $((o_{[i-1]},e),(o_{[i-1]},f))\in A(\Gamma^{[i]})$. It follows that $(e,f)\in A(\Gamma^{i})$ if and only if $((\alpha,e),(\alpha,f))\in A(\Gamma^{[i]})$ for all $\alpha\in S^{i-1}$.

\begin{step}
{\rm $\Gamma^{[i]}=\Gamma^{[i-1]}\square\Gamma^{i}$ for $2\leq i\leq d$.}
\end{step}

Observe that $((\alpha,e),(\beta,f))\in A(\Gamma^{[i]})$ whenever $e=f$ and $((\alpha,e),(\beta,e))\in A(\Gamma^{[i]})$, or $\alpha=\beta$  and $((\alpha,e),(\alpha,f))\in A(\Gamma^{[i]})$ for all $\alpha,\beta\in S^{i-1}$ and $e,f\in S$. By Steps \ref{e=f} and \ref{alpha}, $((\alpha,e),(\beta,f))\in A(\Gamma^{[i]})$ if and only if $e=f$ and $(\alpha,\beta)\in A(\Gamma^{[i-1]})$, or $\alpha=\beta$  and $(e,f)\in A(\Gamma^i)$ for all $\alpha,\beta\in S^{i-1}$ and $e,f\in S$.  Hence, $\Gamma^{[i]}=\Gamma^{[i-1]}\square\Gamma^{i}$.
\end{proof}

We are now ready to prove Theorem \ref{main}.

%
%

\begin{proof}[Proof of Theorem~\ref{main}]
The proof of the sufficiency is straightforward by \cite[Theorem 1.2]{HS04}, \cite[Proposition 2.7]{KSW03} and Proposition \ref{sufficient}. We now prove the necessity.  In view of Theorem \ref{hamming}, we have $a_1=q-2$ and $c_2=2$. Suppose $q=2$. From \cite[Theorem 7.5.2]{AEB98}, $\Sigma$ is distance-transitive. It follows from Proposition \ref{degree*} and Lemma \ref{q=2-1}  that the valency of $\Gamma$ is not more than $2$, which implies that $\Gamma$ is isomorphic to one of the digraphs in Theorem \ref{main} (i) and (ii) by \cite[Theorems 1.2~and 4.1]{HS04}.

We only need to consider the case $q>2$. By  using Lemma \ref{product} repeatedly, we obtain $\Gamma=\Gamma^{1}\square\Gamma^2\square\cdots\square\Gamma^{d}$.
In view of Lemma \ref{degree}, one gets $d^+_{\Gamma_i(o_{[d-1]})}(x)=d^+_{\Gamma_i(o_{[d-1]})}(y)$ for all $x,y\in V(\Gamma_i(o_{[d-1]}))$ with $1\leq i\leq d$.
Note that $\Gamma^i$ is isomorphic to $\Gamma_i(o_{[d-1]})$ for $1\leq i\leq d$. Lemma \ref{proof1} implies that $\Gamma^i$ is a weakly distance-regular digraph of diameter $2$ and girth $g\leq3$ with the same intersection numbers for $1\leq i\leq d$. If $g=3$, then $\Gamma$ is isomorphic to the digraph in Theorem \ref{main} (iv).

Now we suppose $g=2$. Assume the contrary, namely, $d>2$. Let $\wz{\partial}_{\Gamma^1}(o,b)=(1,2)$, $\wz{\partial}_{\Gamma^2}(o,c)=(2,1)$ and $\wz{\partial}_{\Gamma^1}(o,a)=\wz{\partial}_{\Gamma^2}(o,a')=\wz{\partial}_{\Gamma^3}(o,a'')=(1,1)$
for $a,a',a'',b,c\in S$. Note that $\Gamma=\Gamma^1\square\Gamma^2\square\cdots\square\Gamma^{d}$. It follows that $(a,a',a'',o_{[d-3]}),(b,c,o_{[d-2]})\in\Gamma_{3,3}(o_{[d]}).$ One can verify that $$(b,o_{[d-1]})\in P_{(1,2),(2,1)}(o_{[d]},(b,c,o_{[d-2]}))~\mbox{and}~P_{(1,2),(2,1)}(o_{[d]},(a,a',a'',o_{[d-3]}))=\emptyset.$$ This is a contradiction. Thus, $\Gamma$ is isomorphic to one of the digraphs in Theorem \ref{main} (iii).

This completes the proof of  Theorem \ref{main}.
\end{proof}

\section{Proof of Theorem \ref{main3}}

In this section, let   $\Sigma$ be the folded $n$-cube $\square_{n}$ with $n\geq2$.  Note that the induced subgraph of $\Sigma$ on the vertex set $S^{n-2}\times\{a\}$ is a Hamming graph for all $a\in S$.

To prove Theorem \ref{main3}, we need some auxiliary results.

\begin{lemma}\label{k11}
Let $n\geq5$.
Suppose that $x$ and $y$ are antipodal vertices in $\Sigma$. If  $(x,y)\in\Gamma_{1,1}$, then $(u,v)\in\Gamma_{1,1}$ for all antipodal vertices $u,v$.
\end{lemma}
\begin{proof}
Let  $u\in V(\Sigma)$. Since $n\geq5$, there exists a path $(x_0=x,x_1,\ldots,x_l=u)$ in $\Sigma$ such that $x_i$ and $x_{i+1}$ are not antipodal vertices with $0\leq i\leq l-1$. Let $y_{i}$ be the antipodal vertex of $x_{i}$ for $0\leq i\leq l$.  We prove this lemma by induction on $l$. The case that $l=0$ is valid. Suppose $l>0$. By the inductive hypothesis, we have $(x_{l-1},y_{l-1})\in\Gamma_{1,1}$. From Theorem \ref{cube}, one gets $a_1=0$ and $c_2=2$. In view of \cite[Theorem 7.5.2]{AEB98}, $\Sigma$ is distance-transitive. According to Proposition \ref{degree*}, we get $(x_{l-1},x_{l})\in\Gamma_{1,3}\cup\Gamma_{3,1}.$

Let $S=\{o,a\}$. Since $\Sigma$ is distance-transitive, we may assume $(x_{l-1},x_l)=(o_{[n-1]},(a,o_{[n-2]}))$. Then $y_{l-1}=a_{[n-1]}$ and $y_l=(o,a_{[n-2]})$.  Since $n\geq5$, we have  $(y_{l-1},x_l)\in\Sigma_2$.  The fact that $c_2=2$ implies $\Sigma_1(y_{l-1})\cap\Sigma_1(x_l)=\{x_{l-1},y_l\}$.  Note that $x_{l-1}\in P_{(1,1),(1,3)}(y_{l-1},x_{l})\cup P_{(1,1),(3,1)}(y_{l-1},x_{l})$.  By the commutativity of $\Gamma$, we get  $(x_{l},y_{l})\in\Gamma_{1,1}$.  This completes the proof of this lemma.
\end{proof}

\begin{lemma}\label{q=2-2}
%
We have $n\leq6$.
\end{lemma}
\begin{proof}
Let $S=\{o,a\}$. Suppose for the contrary that $n\geq7$.  By Theorem \ref{cube}, we have $a_1=0$ and $c_2=2$. Proposition \ref{degree*} implies $k_{1,1}\in\{0,1\}$.
From    Lemma \ref{k11}, we may assume $(x,y)\in\Gamma_{1,1}$ for all antipodal vertices $x,y$ when  $k_{1,1}=1$.
In view of  Proposition \ref{degree*},  we  assume that $(o_{[n-1]},(a,o_{[n-2]}),(a_{[2]},o_{[n-3]}),(o,a,o_{[n-3]}))$ is a circuit consisting of arcs of type $(1,3)$.

Since $k_{1,1}\in\{0,1\}$ and $n\geq7$, one has $k_{1,3}>1$.
%
%
%
Without loss of generality, we may assume that $(o_{[n-1]},(o_{[2]},a,o_{[n-4]}))\in\Gamma_{1,3}$. Since $\Sigma$ is a folded $n$-cube, we get $((a,o_{[n-2]}),(o_{[2]},a,o_{[n-4]}))\in\Sigma_2$ and $\Sigma_1(a,o_{[n-2]})\cap\Sigma_1(o_{[2]},a,o_{[n-4]})=\{o_{[n-1]},(a,o,a,o_{[n-4]})\}$. Since $o_{[n-1]}\in P_{(3,1),(1,3)}((a,o_{[n-2]}),(o_{[2]},a,o_{[n-4]})),$  one obtains $(a,o,a,o_{[n-4]})\in P_{(1,3),(3,1)}((a,o_{[n-2]}),(o_{[2]},a,o_{[n-4]}))$ from the commutativity of $\Gamma$. By the similar argument, one has $$(u,a,o_{[n-4]})\in P_{(1,3),(3,1)}((u,o_{[n-3]}),(v,a,o_{[n-4]}))$$ for each $(u,v)\in\{(a_{[2]},(a,o)),((o,a),a_{[2]}),(o_{[2]},(o,a))\}.$  Since $$\Sigma_1(o_{[n-1]})\cap\Sigma_1(o,a_{[2]},o_{[n-4]})=\{(o,a,o_{[n-3]}),(o_{[2]},a,o_{[n-4]})\},$$ one gets $(o_{[n-1]},(o,a_{[2]},o_{[n-4]}))\in\Gamma_{4,4}$.

Since $(o,a,o_{[n-3]}),(a_{[3]},o_{[n-4]})\in P_{(1,3),(1,3)}((a_{[2]},o_{[n-3]}),(o,a_{[2]},o_{[n-4]})),$ from Lemma \ref{s=t} (ii), one gets $((a_{[2]},o_{[n-3]}),(o,a_{[2]},o_{[n-4]}))\notin\Gamma_{2,2}.$ By Proposition \ref{degree*}, we have $p_{(2,2),(3,1)}^{(1,3)}\neq0$.
Without loss of generality, we may assume $(a_{[4]},o_{[n-5]})\in P_{(2,2),(3,1)}((a_{[2]},o_{[n-3]}),(a_{[3]},o_{[n-4]})).$ Note that $(o_{[2]},a,o_{[n-4]})\in P_{(1,3),(3,1)}(o_{[n-1]},(o,a_{[2]},o_{[n-4]}))$ and $P_{(1,3),(3,1)}(o_{[n-1]},(a_{[4]},o_{[n-5]}))=\emptyset$. 
Then  $(o_{[n-1]},(a_{[4]},o_{[n-5]}))\notin\Gamma_{4,4}$, and so $n=7$. It follows that $\partial_{\Gamma}(o_{[6]},(a_{[4]},o_{[2]}))=3$ or  $\partial_{\Gamma}((a_{[4]},o_{[2]}),o_{[6]})=3$. According to Theorem \ref{cube} and Proposition \ref{degree*}, we get $k_{1,3}=(7-k_{1,1})/2$, and so $k_{1,1}=1$.

Note that each shortest path of length $3$ from $z$ to $w$   contains an arc of type $(1,1)$ for some $(z,w)\in\{((a_{[4]},o_{[2]}),o_{[6]}),(o_{[6]},(a_{[4]},o_{[2]}))\}$. Since $k_{1,1}=1$, by the commutativity of $\Gamma$, we have $\partial_{\Gamma}((a_{[4]},o_{[2]}),a_{[6]})=2$ or $\partial_{\Gamma}(a_{[6]},(a_{[4]},o_{[2]}))=2$. Without loss of generality, we may assume $((a_{[4]},o_{[2]}),(a_{[5]},o),a_{[6]})$ is a path or $(a_{[6]},(a_{[5]},o),(a_{[4]},o_{[2]}))$ is a path.
Note that $(o_{[6]},(a_{[2]},o_{[4]}))\in\Gamma_{2,2}$. Since $(a_{[6]},o_{[6]})\in\Gamma_{1,1}$, we get $(a_{[6]},(a_{[2]},o_{[4]}))\in\Gamma_{3,3}$, which implies $p^{(3,3)}_{(1,1),(2,2)}\neq0$.

Suppose that $\partial_{\Gamma}((a_{[4]},o_{[2]}),a_{[6]})=2$ and $((a_{[4]},o_{[2]}),(a_{[5]},o),a_{[6]})$ is a path in $\Gamma$. Since $(a_{[6]},o_{[6]},(a,o_{[5]}),(a,o,a,o_{[3]}))$ and $(o_{[6]},(a,o_{[5]}),(a_{[2]},o_{[4]}),(a_{[3]},o_{[3]}))$ are two paths in $\Gamma$, one obtains $\partial_{\Gamma}(a_{[6]},(a,o,a,o_{[3]}))=\partial_{\Gamma}(o_{[6]},(a_{[3]},o_{[3]}))=3$.
Observe that $((a,o,a,o_{[3]}),(a_{[3]},o_{[3]}),(a_{[4]},o_{[2]}),(a_{[5]},o),a_{[6]})$ is a path. It follows that $\partial_{\Gamma}((a,o,a,o_{[3]}),a_{[6]})=3$ or $4$. Note that $(a,o_{[5]}),(o_{[2]},a,o_{[3]})\in P_{(1,3),(1,3)}(o_{[6]},(a,o,a,o_{[3]}))$. If $(a_{[6]},(a,o,a,o_{[3]}))\in\Gamma_{3,3}$, by $p^{(3,3)}_{(1,1),(2,2)}\neq0$ and $k_{1,1}=1$, then $o_{[6]}\in P_{(1,1),(2,2)}(a_{[6]},(a,o,a,o_{[3]}))$, which implies $p^{(2,2)}_{(1,3),(1,3)}\geq2$, contrary to  Lemma \ref{s=t} (ii).
Therefore,  $(a_{[6]},(a,o,a,o_{[3]}))\in\Gamma_{3,4}.$ Since $((a_{[3]},o_{[3]}),(a_{[4]},o_{[2]}),(a_{[5]},o),a_{[6]},o_{[6]})$ is a path, one has $\partial_{\Gamma}((a_{[3]},o_{[3]}),o_{[6]})=3$ or $4$. If $(o_{[6]},(a_{[3]},o_{[3]}))\in\Gamma_{3,3}$, by $p^{(3,3)}_{(1,1),(2,2)}\neq0$ and $k_{1,1}=1$, then $a_{[6]}\in P_{(1,1),(2,2)}(o_{[6]},(a_{[3]},o_{[3]}))$, a contradiction. Hence, $(o_{[6]},(a_{[3]},o_{[3]}))\in\Gamma_{3,4}$. But $P_{(1,1),(2,s)}(o_{[6]},(a_{[3]},o_{[3]}))=\emptyset$ and $o_{[6]}\in P_{(1,1),(2,s)}(a_{[6]},(a,o,a,o_{[3]}))$ with $s=\partial_{\Gamma}((a,o,a,o_{[3]}),o_{[6]})$, a contradiction. Thus, $\partial_{\Gamma}((a_{[4]},o_{[2]}),a_{[6]})\neq2$, and so $\partial_{\Gamma}((a_{[4]},o_{[2]}),o_{[6]})\neq3$.


Note that $(a_{[6]},(a_{[5]},o),(a_{[4]},o_{[2]}))$ is a path and $(o_{[6]},(a_{[4]},o_{[2]}))\in\Gamma_{3,4}$. Observe that $((a_{[5]},o),(a_{[4]},o_{[2]}))\in A(\Gamma)$ and $((a_{[4]},o_{[2]}),(a_{[2]},o_{[4]}))\in\Gamma_{2,2}$. Then $\partial_{\Gamma}((a_{[5]},o),(a_{[2]},o_{[4]}))=3$.  Since $((a_{[2]},o_{[4]}),(o,a,o_{[4]}),o_{[6]},a_{[6]},(a_{[5]},o))$ is a path, we obtain  $\partial_{\Gamma}((a_{[2]},o_{[4]}),(a_{[5]},o))=3$ or $4$. If $((a_{[5]},o),(a_{[2]},o_{[4]}))\in\Gamma_{3,3}$, by $p^{(3,3)}_{(1,1),(2,2)}\neq0$ and $k_{1,1}=1$, then  $(o_{[5]},a)\in P_{(1,1),(2,2)}((a_{[5]},o),(a_{[2]},o_{[4]}))$,  a contradiction.  Thus, $((a_{[5]},o),(a_{[2]},o_{[4]}))\in\Gamma_{3,4}.$  But $P_{(1,1),(2,t)}((a_{[5]},o),(a_{[2]},o_{[4]}))=\emptyset$ and  $a_{[6]}\in P_{(1,1),(2,t)}(o_{[6]},(a_{[4]},o_{[2]}))$ with $t=\partial_{\Gamma}((a_{[4]},o_{[2]}),a_{[6]})$, which is impossible. This completes the proof of this lemma.
%
\end{proof}

Now we are ready to prove Theorem~\ref{main3}.

\begin{proof}[Proof of Theorem~\ref{main3}]
Note that $n\geq2$. According to Lemma \ref{q=2-2}, we obtain $n\leq6$. If $n=6$, from Theorem \ref{hamming}, \cite[Theorem 7.5.2]{AEB98} and Proposition \ref{degree*}, then $k_{1,3}=3$, contrary to \cite[Theorem 1.1]{YYF18}. Therefore, $n\leq5$. By  \cite{H}, $\Gamma$ is isomorphic to {\rm Cay}$(\mathbb{Z}_4,\{1,2\})$. Thus, $n=2$.
\end{proof}

\section{Proof of Theorem \ref{main2}}


In this section, without loss of generality, we may assume $S=\mathbb{Z}_4$, where $\mathbb{Z}_{q}=\{0,1,\ldots,q-1\}$.  Let  $\Sigma$ be the Doob graph  $H(d-2d_1,4)\,\square\underbrace{\Lambda\,\square\,\Lambda\,\square\ldots\square\,\Lambda}_{d_1}$, where $\Lambda\simeq{\rm Cay}(\mathbb{Z}_4^2,\{(\pm1,0),(0,\pm1),\pm(1,1)\})$ and $d\geq2d_1>0$.


To prove Theorem \ref{main2}, we need some lemmas.

\begin{lemma}\label{a1}
Let  $(x,z)\in\Sigma_1$ and $\Sigma_1(x)\cap\Sigma_1(z)=\{y,y'\}$. If $\tilde{\partial}_{\Gamma}(x,y)\neq\tilde{\partial}_{\Gamma}(y',z)$, then $\tilde{\partial}_{\Gamma}(y,z)=\tilde{\partial}_{\Gamma}(x,y)$ and $\tilde{\partial}_{\Gamma}(x,y')=\tilde{\partial}_{\Gamma}(y',z)$.
%
%
\end{lemma}
\begin{proof}
Let   $y\in P_{(a,b),(a',b')}(x,z)$. Suppose $(a,b)\neq(a',b')$. Since $\Sigma_1(x)\cap\Sigma_1(z)=\{y,y'\}$ and  $\Gamma$ is commutative, we get $y'\in P_{(a',b'),(a,b)}(x,z)$,  contrary to the fact that $\tilde{\partial}_{\Gamma}(x,y)\neq\tilde{\partial}_{\Gamma}(y',z)$. Thus, $(a,b)=(a',b')$. Similarly,  $\tilde{\partial}_{\Gamma}(x,y')=\tilde{\partial}_{\Gamma}(y',z)$.
%
\end{proof}

Set $d_2=d-2d_1$.
For each $j\in\{1,2,\ldots,d_1\}$ and $a_i\in \mathbb{Z}_{4}$ with $1\leq i\leq d-2$, denote $\Delta_j(a_{1},a_{2},\ldots,a_{d-2})$ the induced subdigraph of $\Gamma$ on the set
\begin{align}
\{(a_1,a_2,\ldots,a_{d_2+2j-2},a,b,a_{d_2+2j-1},a_{d_2+2j},\ldots,a_{d-2})\mid a,b\in \mathbb{Z}_{4}\}.\nonumber
\end{align}

\begin{lemma}\label{a1-2}
Let $x,z\in V(\Delta_{j}(\alpha))$ for $1\leq j\leq d_1$ and $\alpha\in\mathbb{Z}_{4}^{d-2}$. Suppose that $(x,z)\in\Sigma_1$ and $\Sigma_1(x)\cap\Sigma_1(z)=\{y,y'\}$.   If $(x,y)\in\Gamma_{1,p-1}$ and $(z,y')\in\Gamma_{1,t-1}$ for $(p,t)\neq(2,2)$, then $p=t\leq4$ and $p_{(1,p-1),(1,p-1)}^{(2,2)}\neq0$.
%
%
\end{lemma}
\begin{proof}
Note that $\tilde{\partial}_{\Gamma}(x,y)\neq\tilde{\partial}_{\Gamma}(y',z)$. In view of Lemma \ref{a1}, we have  $(y,z)\in\Gamma_{1,p-1}$ and $(y',x)\in\Gamma_{1,t-1}$. Since $(y,z,y',x)$ is a circuit in $\Gamma$, we get $p\leq4$ and $P_{(1,p-1),(1,t-1)}(y,y')\neq\emptyset$.  Note that the underlying graph of $\Delta_j(\alpha)$ is isomorphic to ${\rm Cay}(\mathbb{Z}_4^2,\{(\pm1,0),(0,\pm1),\pm(1,1)\})$.  Then $y,y'\in V(\Delta_j(\alpha))$ and $(y,y')\in\Sigma_2$. It follows that  $(y,y')\in\Gamma_{2,2}$,  and so $p^{(2,2)}_{(1,p-1),(1,t-1)}\neq0$. By Theorem \ref{Doob}, we get $c_2=2$.   Lemma \ref{s=t} implies $p=t$. This completes the proof of this lemma.
\end{proof}

The proof of the following result relies on Lemma \ref{a1-2}.

\begin{lemma}\label{arc}
Let $x_0,x_1\in V(\Delta_{j}(\alpha))$ for $1\leq j\leq d_1$ and $\alpha\in\mathbb{Z}_{4}^{d-2}$. If $(x_0,x_1)\in\Gamma_{1,p-1}$ for $p\geq2$, then $p\leq4$ and $p^{(2,2)}_{(1,p-1),(1,p-1)}\neq0$.
\end{lemma}
\begin{proof}
Fix $j\in\{1,2,\ldots,d_1\}$ and $\alpha\in\mathbb{Z}_4^{d-2}$.    Let $\varphi$ be an isomorphism from ${\rm Cay}(\mathbb{Z}_4^2,\{(\pm1,0),(0,\pm1),\pm(1,1)\})$ to $\Delta_j(\alpha)$. Without loss of generality, we may assume $x_0=\varphi(0,0)$. It follows that $x_1\in\{\varphi(\pm1,0),\varphi(0,\pm1),\pm\varphi(1,1)\}$. Since the proofs are similar, we assume $x_1=\varphi(1,0)$. Let $x_i=\varphi(i,0)$ for  $i\in\{2,3\}$. It follows that $(x_i,x_{i+1})\in\Sigma_{1}$ for  $i\in\mathbb{Z}_4$ and $(x_0,x_2)\in\Sigma_2$. In view of Theorem \ref{Doob}, we have $a_1=2$ and $c_2=2$. Next we divide this proof into two cases according to whether $(x_0,x_1,x_2,x_3)$ is a circuit in $\Gamma$.

\textbf{Case 1.} $(x_0,x_1,x_2,x_3)$ is a circuit in $\Gamma$.

Note that $p\leq4$ and $(x_0,x_2)\in\Gamma_{2,2}$. The fact  $x_1\in P_{(1,p-1),(1,s-1)}(x_0,x_2)$ for some $s\geq2$ implies $p^{(2,2)}_{(1,p-1),(1,s-1)}\neq0$. Since $c_2=2$, by Lemma \ref{s=t}, we get $p=s$, and so $p^{(2,2)}_{(1,p-1),(1,p-1)}\neq0$.

\textbf{Case 2.} $(x_0,x_1,x_2,x_3)$ is not a circuit in $\Gamma$.

Observe that $(x_1,x_2),(x_2,x_3)$ or $(x_3,x_0)$ is not an arc in $\Gamma$.
According to the commutativity of $\Gamma$,  we may assume $(x_3,x_0)\in\Gamma_{l-1,1}$ for  $l>2$. Set  $y_i=\varphi(i,1)$ for $i\in\{0,1\}$.  The fact that $\Sigma$ is a Doob graph implies  $(y_0,y_1)\in\Sigma_1$.
Then $(y_0,y_1)\in\Gamma_{t-1,1}$ or $\Gamma_{1,t-1}$ with $t\geq2$.

\textbf{Case 2.1.} $(y_0,y_1)\in\Gamma_{t-1,1}$ with $t>2$.

Since $\Sigma$ is a Doob graph, one has $(x_0,y_1)\in\Sigma_1$ and $\Sigma_1(x_0)\cap\Sigma_1(y_1)=\{x_1,y_0\}$.   Since  $(x_0,x_1)\in\Gamma_{1,p-1}$, by Lemma \ref{a1-2}, we have $p=t\leq4$ and $p^{(2,2)}_{(1,p-1),(1,p-1)}\neq0$.

\textbf{Case 2.2.} $(y_0,y_1)\in\Gamma_{1,t-1}$ with $t\geq2$.

Since $\Sigma$ is a Doob graph, we get $(x_0,y_0)\in\Sigma_1$ and $\Sigma_1(x_0)\cap\Sigma_1(y_0)=\{x_3,y_1\}$. Note that $(x_0,x_3)\in\Gamma_{1,l-1}$ with $l>2$.  
It follows from  Lemma \ref{a1-2} that  $l=t\leq4$ and $p^{(2,2)}_{(1,l-1),(1,l-1)}\neq0$. By Lemma \ref{a1}, we obtain $(x_3,y_0)\in\Gamma_{1,l-1}$. Observe that $(x_0,y_1)\in\Sigma_1$ and  $\Sigma_1(x_0)\cap\Sigma_1(y_1)=\{x_1,y_0\}$. If $p\neq l$, from Lemma \ref{a1}, then 
$(x_0,y_0)\in\Gamma_{1,l-1}$ since $\tilde{\partial}_{\Gamma}(x_0,x_1)\neq\tilde{\partial}_{\Gamma}(y_0,y_1)$, which implies $p^{(1,l-1)}_{(1,l-1),(l-1,1)}\neq0$ by $y_0\in P_{(1,l-1),(l-1,1)}(x_0,x_3)$, contrary to Lemma \ref{0}. Thus, $p=l$, and so $p\leq4$ and $p^{(2,2)}_{(1,p-1),(1,p-1)}\neq0$.

Since $j$ and $\alpha$ were arbitrary,  the desired result follows.
\end{proof}


\begin{lemma}\label{K4-2}
Exactly one of the  following holds:
\begin{itemize}
\item[{\rm(i)}] $p^{(1,1)}_{(2,2),(1,1)}\neq0$, $p^{(1,1)}_{(1,1),(1,1)}=2$ and $\Delta_j(\alpha)$ is isomorphic to the digraph ${\rm Cay}(\mathbb{Z}_4^2,\{(\pm1,0),(0,\pm1),\pm(1,1)\})$ for $j\in\{1,2,\ldots,d_1\}$ and $\alpha\in\mathbb{Z}^{d-2}_{4}$;

\item[{\rm(ii)}] $p^{(1,2)}_{(2,2),(2,1)}\neq0$, $p^{(2,1)}_{(1,2),(1,2)}=2$ and $\Delta_j(\alpha)$ is isomorphic to the digraph ${\rm Cay}(\mathbb{Z}_4^2,\{(1,0),(0,1),(-1,-1)\})$ for  $j\in\{1,2,\ldots,d_1\}$ and $\alpha\in\mathbb{Z}^{d-2}_{4}$.
\end{itemize}
\end{lemma}
\begin{proof}
In view of Theorem \ref{Doob}, we get $a_1=2$ and $c_2=2$. Note that the underlying graph of $\Delta_j(\alpha)$ is isomorphic to ${\rm Cay}(\mathbb{Z}_4^2,\{(\pm1,0),(0,\pm1),\pm(1,1)\})$. Then $\Sigma_1(u)\cap\Sigma_1(v)\subseteq V(\Delta_j(\alpha))$ for all $u,v\in V(\Delta_j(\alpha))$ with $u\neq v$.

Let $x,y\in V(\Delta_{j}(\alpha))$ such that $(x,y)\in\Gamma_{1,p-1}$ with $p\geq2$. According to Lemma \ref{arc}, we have $p\leq4$ and $p^{(2,2)}_{(1,p-1),(1,p-1)}\neq0$. It follows from Lemma \ref{intersection numners} (ii) that $p^{(1,p-1)}_{(2,2),(p-1,1)}\neq0$.   By Lemmas \ref{pneqt} and \ref{arc}, we have $(u,v)\in\Gamma_{1,p-1}$  for all $(u,v)\in A(\Delta_j(\alpha))$, $j\in\{1,2,\ldots,d_1\}$ and $\alpha\in\mathbb{Z}^{d-2}_{4}$.

Suppose $p=2$. Then   $\Delta_j(\alpha)\simeq{\rm Cay}(\mathbb{Z}_4^2,\{(\pm1,0),(0,\pm1),\pm(1,1)\})$ for all $j\in\{1,2,\ldots,d_1\}$ and $\alpha\in\mathbb{Z}^{d-2}_{4}$, and $p^{(1,1)}_{(2,2),(1,1)}\neq0$. Since $a_1=2$ and $\Sigma_1(u)\cap\Sigma_1(v)\subseteq V(\Delta_j(\alpha))$ for all $u,v\in V(\Delta_j(\alpha))$ with $u\neq v$, we have $p^{(1,1)}_{(1,1),(1,1)}=2$. Thus, (i) is valid.

Suppose $p\in\{3,4\}$.   By Lemma \ref{0} and the commutativity of $\Gamma$, one has $p^{(1,p-1)}_{(1,p-1),(p-1,1)}=p^{(1,p-1)}_{(p-1,1),(1,p-1)}=0$. In view of Lemma \ref{intersection numners} (ii), we get $p^{(1,p-1)}_{(1,p-1),(1,p-1)}=0$. Since $a_1=2$ and $\Sigma_1(u)\cap\Sigma_1(v)\subseteq V(\Delta_j(\alpha))$ for all $u,v\in V(\Delta_j(\alpha))$ with $u\neq v$, we obtain  $p^{(1,p-1)}_{(p-1,1),(p-1,1)}=2$, and so $p=3$. It follows that $p^{(1,2)}_{(2,2),(2,1)}\neq0$ and $p^{(1,2)}_{(2,1),(2,1)}=2$. Lemma \ref{intersection numners} (ii) implies $p^{(2,1)}_{(1,2),(1,2)}=2$.

Fix $j\in\{1,2,\ldots,d_1\}$ and $\alpha\in\mathbb{Z}_4^{d-2}$.    Let $\varphi$ be an isomorphism from ${\rm Cay}(\mathbb{Z}_4^2,\{(\pm1,0),(0,\pm1),\pm(1,1)\})$ to $\Delta_j(\alpha)$. Set $x_{i,l}=\varphi(i,l)$ for $i,j\in\mathbb{Z}_4$. Since $\Sigma$ is a Doob graph, one gets $(x_{0,0},x_{1,1})\in\Sigma_1$. Without loss of generality, we may assume  $(x_{0,0},x_{1,1})\in\Gamma_{2,1}$.  Since  $p^{(2,1)}_{(1,2),(1,2)}=2$ and $\Sigma_1(x_{0,0})\cap\Sigma_1(x_{1,1})=\{x_{1,0},x_{0,1}\}$, we obtain  $P_{(1,2),(1,2)}(x_{0,0},x_{1,1})=\{x_{1,0},x_{0,1}\}$. Note that $\Sigma_1(x_{1,0})\cap\Sigma_1(x_{1,1})=\{x_{0,0},x_{2,1}\}.$ Since $p^{(2,1)}_{(1,2),(1,2)}=2$, we have $x_{2,1}\in P_{(1,2),(1,2)}(x_{1,1},x_{1,0})$. The fact $\Sigma_1(x_{1,0})\cap\Sigma_1(x_{2,1})=\{x_{2,0},x_{1,1}\}$  implies $x_{2,0}\in P_{(1,2),(1,2)}(x_{1,0},x_{2,1})$. By the similar argument, we get  $(x_{i,0},x_{i+1,1})\in\Gamma_{2,1}$ and $P_{(1,2),(1,2)}(x_{i,0},x_{i+1,1})=\{x_{i+1,0},x_{i,1}\}$   for all $i\in\mathbb{Z}_4$.

Note that $(x_{1,1},x_{0,1})\in\Gamma_{2,1}$ and $\Sigma_1(x_{1,1})\cap\Sigma_1(x_{0,1})=\{x_{0,0},x_{1,2}\}.$ The fact that $p^{(2,1)}_{(1,2),(1,2)}=2$ implies $x_{1,2}\in P_{(1,2),(1,2)}(x_{1,1},x_{0,1})$.  Since $\Sigma_1(x_{0,1})\cap\Sigma_1(x_{1,2})=\{x_{0,2},x_{1,1}\},$  we get $x_{0,2}\in P_{(1,2),(1,2)}(x_{0,1},x_{1,2})$. By the similar argument, $x_{i+1,i'},x_{i,i'+1},x_{i-1,i'-1}\in\Gamma_{1,2}(x_{i,i'})$ for all $i,i'\in\mathbb{Z}_4$.
%
Hence, $\Delta_j(\alpha)\simeq{\rm Cay}(\mathbb{Z}_4^2,\{(1,0),(0,1),(-1,-1)\})$.
Since $j$ and $\alpha$ were arbitrary, (ii) is valid.
\end{proof}

Note that the induced subgraph of $\Sigma$ on the vertex set $\mathbb{Z}_4^{d_2}\times\{\gamma\}$ is a Hamming graph for all $\gamma\in\mathbb{Z}_4^{2d_1}$. For each $i\in\{1,2,\ldots,d_2\}$ and $a_j\in \mathbb{Z}_{4}$ with $1\leq j\leq d-1$, recall that $\Gamma_{i}(a_{1},a_{2},\ldots,a_{d-1})$ is the induced subdigraph of $\Gamma$ on the set
\begin{align}
\{(a_1,a_2,\ldots,a_{i-1},b,a_{i},a_{i+1},\ldots,a_{d-1})\mid b\in \mathbb{Z}_{4}\}.\nonumber
\end{align}

\begin{lemma}\label{d2=0}
We have $d_2=0$.
\end{lemma}
\begin{proof}
Assume the contrary, namely, $d_2>0$. We claim that $\Gamma_{i}(\beta)\simeq K_4$ or {\rm Cay}$(\mathbb{Z}_4,\{1,2\})$ for $1\leq i\leq d_2$ and $\beta\in\mathbb{Z}^{d-1}_{4}$.    If $(u,v)\in\Gamma_{1,1}$ for all $u,v\in V(\Gamma_{i}(\beta))$ with $u\neq v$, then $\Gamma_{i}(\beta)\simeq K_4$. Now we suppose $(u,v)\in\Gamma_{1,p-1}$ with $p>2$ for some $u,v\in V(\Gamma_{i}(\beta))$.  In view of Theorem \ref{Doob}, we have $a_1=2$ and $c_2=2$. By Lemma \ref{K4-2}, we get $p^{(1,1)}_{(2,2),(1,1)}\neq0$ or $p^{(1,2)}_{(2,2),(2,1)}\neq0$. It follows from Lemma \ref{pneqt} that $p^{(1,s-1)}_{(2,2),(s-1,1)}=0$ for all $s\notin\{2,3\}$.  If $d^+_{\Gamma_i(\beta)}(x)\neq d^+_{\Gamma_i(\beta)}(y)$ for some $x,y\in V(\Gamma_i(\beta))$, by Proposition \ref{degree-main} and Lemma \ref{p=4}, then $p^{(1,3)}_{(2,2),(3,1)}\neq0$, a contradiction.   Thus, $d^+_{\Gamma_i(\beta)}(x)=d^+_{\Gamma_i(\beta)}(y)$ for all $x,y\in V(\Gamma_i(\beta))$.
From Lemma \ref{proof1}, $\Gamma_i(\beta)$ is  weakly distance-regular digraph of diameter $2$ and girth $g\in\{2,3\}$. By  \cite{H}, $\Gamma_{i}(\beta)$ is isomorphic to  {\rm Cay}$(\mathbb{Z}_4,\{1,2\})$. Thus, our claim holds.

Suppose  $\Gamma_{i}(\beta)\simeq{\rm Cay}(\mathbb{Z}_4,\{1,2\})$ for some $i\in\{1,2,\ldots,d_2\}$ and $\beta\in\mathbb{Z}_4^{d-1}$.    Let $\varphi$ be an isomorphism from ${\rm Cay}(\mathbb{Z}_4,\{1,2\})$ to $\Gamma_i(\beta)$. Set $x_{l}=\varphi(l)$ for $l\in\mathbb{Z}_4$. Then $(x_0,x_2)\in\Gamma_{1,1}$ and $x_1\in P_{(1,2),(1,2)}(x_0,x_2)$. It follows that $p^{(1,1)}_{(1,2),(1,2)}\neq0$.  Lemma \ref{intersection numners} (ii) implies $p^{(2,1)}_{(1,1),(1,2)}\neq0$. Let $j\in\{1,2,\ldots,d_1\}$ and $\alpha\in\mathbb{Z}_4^{d-2}$.   By Lemma \ref{K4-2}, if $\Delta_j(\alpha)$ is isomorphic to ${\rm Cay}(\mathbb{Z}_4^2,\{(\pm1,0),(0,\pm1),\pm(1,1)\})$, then $p^{(1,1)}_{(1,1),(1,1)}=2$, which implies $p^{(1,1)}_{(1,2),(1,2)}=0$ since $a_1=2$, a contradiction; if $\Delta_j(\alpha)$ is isomorphic to ${\rm Cay}(\mathbb{Z}_4^2,\{(1,0),(0,1),(-1,-1)\})$, then $p^{(2,1)}_{(1,2),(1,2)}=2$, which implies $p^{(2,1)}_{(1,1),(1,2)}=0$, a contradiction. Hence, $\Gamma_{i}(\beta)\simeq K_4$ for all $1\leq i\leq d_2$ and $\beta\in\mathbb{Z}_4^{d-1}$.

Note that $\Gamma$ is not an undirected graph. By Lemma \ref{K4-2}, there exists an isomorphism $\varphi$ from  ${\rm Cay}(\mathbb{Z}_4^2,\{(1,0),(0,1),(-1,-1)\})$ to $\Delta_1(0_{[d-2]})$.
Without loss of generality, we may assume  $\varphi(i,j)=(0_{[d_2]},i,j,0_{[2d_1-2]})$ for $i,j\in\mathbb{Z}_4$. Then $(0_{[d]},(0_{[d_2]},2,1,0_{[2d_1-2]}))\in\Gamma_{3,3}$ and $(0_{[d]},(0_{[d_2]},2,0_{[2d_1-1]}))\in\Gamma_{2,2}$. Since $(0_{[d]},(1,0_{[d-1]}))\in\Gamma_{1,1}$ and $(0_{[d]},(1,0_{[d_2-1]},2,0_{[2d_1-1]}))\in\Sigma_3$, one obtains $(0_{[d]},(1,0_{[d_2-1]},2,0_{[2d_1-1]}))\in\Gamma_{3,3}$. But $P_{(1,2),(2,1)}(0_{[d]},(1,0_{[d_2-1]},2,0_{[2d_1-1]}))=\emptyset$ and $(0_{[d_2]},1,0_{[2d_1-1]})\in P_{(1,2),(2,1)}(0_{[d]},(0_{[d_2]},2,1,0_{[2d_1-2]}))$, a contradiction.
\end{proof}

Now we are ready to prove Theorem \ref{main2}.

\begin{proof}[Proof of Theorem~\ref{main2}]
The proof of the sufficiency is straightforward by \cite[Theorem 1.1]{YYF18}. We now prove the necessity.   In view of Lemma \ref{d2=0}, we have $d_2=0$. Since $\Gamma$ is not an undirected graph, by Lemma \ref{K4-2}, $\Delta_j(\alpha)$ is isomorphic to  ${\rm Cay}(\mathbb{Z}_4^2,\{(1,0),(0,1),(-1,-1)\})$ for all $j\in\{1,2,\ldots,d\}$ and $\alpha\in\mathbb{Z}^{d-2}_{4}$. It follows that each arc in $\Gamma$ is of type $(1,2)$.
Without loss of generality, we may assume that  $\Delta_1(0_{[d-2]})={\rm Cay}(\mathbb{Z}_4^d,\{(1,0,0_{[d-2]}),(0,1,0_{[d-2]}),(-1,-1,0_{[d-2]})\}).$ Note that $(-1,0,0_{[d-2]})\in P_{(2,1),(2,2)}(0_{[d]},(1,2,0_{[d-2]}))$ and $(0_{[d]},(1,2,0_{[d-2]}))\in\Gamma_{3,3}$. Then $p^{(3,3)}_{(2,1),(2,2)}\neq0$.

Assume the contrary, namely, $d_1>1$. Note that $(0_{[d]},z_1)\in\Gamma_{1,2}$ with $z_1=(1,0,0_{[d-2]})$. Without loss of generality, we may assume   $(z_1,x)\in\Gamma_{2,1}$, where $x=(1,0,a,b,0_{[d-4]})$ for some $a,b\in\mathbb{Z}_4$.  Note that $z_1\in P_{(1,2),(2,1)}(0_{[d]},x)$ and $(0_{[d]},x)\in\Sigma_2$. Since $c_2=2$, we get $\Sigma_1(0_{[d]})\cap\Sigma_1(x)=\{z_1,z_2\}$ with $z_2=(0_{[2]},a,b,0_{[d-4]})$. The commutativity of $\Gamma$ implies $z_2\in P_{(2,1),(1,2)}(0_{[d]},x)$.  Since $c_2=2$,  one obtains $P_{(1,r-1),(1,s-1)}(0_{[d]},x)=P_{(r-1,1),(s-1,1)}(0_{[d]},x)=\emptyset$ for all $r,s\geq2$. Then $(0_{[d]},x)\in\Gamma_{3,3}$.

Since $p^{(3,3)}_{(2,1),(2,2)}\neq0$,  there exists  $y\in P_{(2,1),(2,2)}(0_{[d]},x)$.
It follows that $(0_{[d]},y)\in\Sigma_1$ and  $(y,x)\in\Sigma_2$. Then $y\in V(\Delta_{j'}(0_{[d-2]}))$ for some $j'\in\{1,2\}$. Suppose  $j'=2$. Then $y=(0_{[2]},a',b',0_{[d-4]})$ for some $(a',b')\neq(0,0)$.   Since  $(y,x)\in\Sigma_2$ and $c_2=2$, we have $\Sigma_1(x)\cap\Sigma_1(y)=\{z_2,(1,0,a',b',0_{[d-4]})\}$. The fact that each arc in $\Gamma$ is of type $(1,2)$ implies  $(y,z_2)\in\Gamma_{1,2}\cup\Gamma_{2,1}$. Since $0_{[d]}\in P_{(1,2),(2,1)}(y,z_2)$, we have $p^{(1,2)}_{(1,2),(2,1)}\neq0$, contrary to Lemma \ref{0}. Thus, $j'=1$, and so $y=(a',b',0_{[d-2]})$ for some $(a',b')\neq(0,0)$. Since $y\in\Gamma_{2,1}(0_{[d]})$, one has $(a',b')\neq(1,0)$.

Since $(y,x)\in\Sigma_2$ and $c_2=2$, one obtains $\Sigma_1(y)\cap\Sigma_1(x)=\{z_1,w\}$ with $w=(a',b',a,b,0_{[d-4]})$. According to Lemma \ref{K4-2} (ii), we have  $p^{(1,2)}_{(2,2),(2,1)}\neq0$. By Lemma \ref{intersection numners} (ii) and Lemma \ref{s=t}, one gets $p^{(2,2)}_{(1,2),(1,2)}=p^{(2,2)}_{(2,1),(2,1)}=1$. Since $(y,x)\in\Gamma_{2,2}$ and $(z_1,x)\in\Gamma_{2,1}$, we obtain $P_{(2,1),(2,1)}(y,x)=\{z_1\}$ and $P_{(1,2),(1,2)}(y,x)=\{w\}$. Note that $0_{[d]}\in\Sigma_1(z_2)\cap\Sigma_1(y)$. Since $z_2=(0_{[2]},a,b,0_{[d-4]})$ and $y=(a',b',0_{[d-2]})$, we have $(z_2,y)\in\Sigma_2$. The fact that $c_2=2$ implies $\Sigma_1(z_2)\cap\Sigma_1(y)=\{0_{[d]},w\}$. Note that  each arc in $\Gamma$ is of type $(1,2)$. It follows that $(z_2,w)\in\Gamma_{1,2}\cup\Gamma_{2,1}$. Since $x\in P_{(1,2),(2,1)}(w,z_2)$, we get $p^{(1,2)}_{(1,2),(2,1)}\neq0$, contrary to Lemma \ref{0}.

Therefore, $d_1=1$, the desired result follows.
\end{proof}


\section*{Acknowledgements}

The authors would like to thank the anonymous reviewers for their careful reading of the manuscript of the paper and their invaluable suggestions. Y.~Yang is supported by NSFC (12101575) and the Fundamental Research Funds for the Central Universities (2652019319),  K. Wang is supported by the National Key R$\&$D Program of China (No.~2020YFA0712900) and NSFC (12071039, 12131011).

\section*{Data Availability Statement}

Data sharing not applicable to this article as no datasets were generated or analysed during the current study.

\section*{Conflict of Interests}

The authors declare that they have no conflict of interest.

\end{CJK*}

\end{document}